%% file: double_dinaturality.tex
\documentclass[letterpaper]{article} 

\usepackage{import} 

\import{setup/}{setup}

\import{macros/}{macros}

\import{macros/}{local_macros}

\providecommand{\keywords}[1]{\textbf{Keywords: \,} #1}
\providecommand{\msc}[1]{\textbf{Mathematics Subject Classification: \,} #1}

\hypersetup
{
	pdfauthor = {Edward Morehouse} ,
	pdftitle = {Dinaturality for Double Categories} ,
	pdfcreator = {} , 
	pdfproducer = { } , 
}

\title{Dinaturality for Double Categories}

\author
{
	\small
	Edward Morehouse%
}
\date{}

\begin{document}
	\maketitle
	\subimport{content/}{outline}

\end{document}

%% file: setup/setup.tex

\usepackage{iftex} 
\usepackage{xparse} 
\usepackage{xifthen} 
\usepackage{xstring} 

\import{./}{docstyle} 
\import{./}{fontstyle} 
\import{./}{mathstyle} 
\import{./}{theoremstyle} 
\import{./}{codestyle} 
\import{./}{tikzstyle} 
\import{./}{bibstyle} 
\import{./}{hyperstyle} 

%% file: setup/docstyle.tex
\usepackage{newsec} 
\usepackage{parskip} 
\usepackage{verbatim} 
\usepackage{todonotes} 
\usepackage{comment} 
\usepackage{enumerate} 
\usepackage{multicol} 
\usepackage{authblk} 
\usepackage[hmargin=36mm , vmargin=36mm]{geometry} 

%% file: setup/fontstyle.tex
\ifpdftex
	\relax
\else
	\usepackage{fontspec}
	\defaultfontfeatures{Mapping = tex-text , Ligatures = NoCommon , Scale = MatchUppercase}
	\defaultfontfeatures[\ttfamily]{Scale = MatchLowercase}
\fi

%% file: setup/mathstyle.tex
\usepackage{amssymb}
\usepackage{amsmath}
	\numberwithin{equation}{section}
\usepackage{mathtools}
\usepackage{scalerel} 
\usepackage{stackengine} 
\usepackage{proof}
	\inferLineSkip=3pt 
	\inferLabelSkip=8pt 

\ifpdftex
	\usepackage{latexsym} 
	\usepackage{lmodern} 
	\usepackage{frenchmath} 
	\usepackage{upgreek} 
	\usepackage{mathabx} 
	\usepackage{mathbbol} 
\else
	\usepackage{unicode-math} 
	\unimathsetup
	{
		math-style = french , 
		partial = upright ,
		nabla = upright
	}
	\setmathfont{latinmodern-math.otf} 
\fi

\ifpdftex
	\usepackage{newunicodechar} 
	\newunicodechar{′}{{^{\prime}}}
	\newunicodechar{¨}{\text{\textasciidieresis}}
	\newunicodechar{≠}{\neq}
	\newunicodechar{−}{-}
	\newunicodechar{×}{\times}
	\newunicodechar{∧}{\wedge}
	\newunicodechar{∨}{\vee}
	\newunicodechar{⊕}{\oplus}
	\newunicodechar{⊖}{\ominus}
	\newunicodechar{⊗}{\otimes}
	\newunicodechar{⊙}{\odot}
	\newunicodechar{⊛}{\oasterisk}
	\newunicodechar{⋅}{\cdot}
	\newunicodechar{∘}{\circ}
	\newunicodechar{∙}{\bullet}
	\newunicodechar{∗}{\ast}
	\newunicodechar{°}{{}^\circ}
	\newunicodechar{⋆}{\star}
	\newunicodechar{∞}{\infty}
	\newunicodechar{∅}{\varnothing}
	\newunicodechar{⋯}{\cdots}
	\newunicodechar{…}{\ldots}
	\newunicodechar{→}{\rightarrow}
	\newunicodechar{←}{\leftarrow}
	\newunicodechar{⇸}{\stackinset{c}{-0.125ex}{c}{0.25ex}{$\shortmid$}{$\rightarrow$}}
	\newunicodechar{⇷}{\stackinset{c}{0.125ex}{c}{0.25ex}{$\shortmid$}{$\leftarrow$}}
	\newunicodechar{⤍}{\dashrightarrow}
	\newunicodechar{⤌}{\dashleftarrow}
	\newunicodechar{⇒}{\Rightarrow}
	\newunicodechar{⇐}{\Leftarrow}
	\newunicodechar{↓}{\downarrow}
	\newunicodechar{↑}{\uparrow}
	\newunicodechar{⇑}{\Uparrow}
	\newunicodechar{⇓}{\Downarrow}
	\newunicodechar{↔}{\leftrightarrow}
	\newunicodechar{↦}{\mapsto}
	\newunicodechar{◇}{\Diamond}
	\newunicodechar{□}{\square}
	\newunicodechar{⟨}{\langle}
	\newunicodechar{⟩}{\rangle}
	\newunicodechar{⌜}{\ulcorner}
	\newunicodechar{⌝}{\urcorner}
	\newunicodechar{⌞}{\llcorner}
	\newunicodechar{⌟}{\lrcorner}
	\newunicodechar{∫}{\int}
	\newunicodechar{∏}{\prod}
	\newunicodechar{∐}{\coprod}
	\newunicodechar{∑}{\sum}
	\newunicodechar{≅}{\cong}
	\newunicodechar{≃}{\simeq}
	\newunicodechar{∼}{\sim}
	\newunicodechar{≔}{\coloneqq}
	\newunicodechar{≕}{\eqqcolon}
	\newunicodechar{∈}{\in}
	\newunicodechar{∋}{\ni}
	\newunicodechar{∀}{\forall}
	\newunicodechar{∃}{\exists}
	\newunicodechar{⊤}{\top}
	\newunicodechar{⊥}{\bot}
	\newunicodechar{⊢}{\vdash}
	\newunicodechar{⊣}{\dashv}
	\newunicodechar{∆}{\Delta} 
	\newunicodechar{∇}{\nabla}
	\newunicodechar{∂}{\partial}
	\newunicodechar{∞}{\infty}
	\newunicodechar{α}{\alpha}
	\newunicodechar{β}{\beta}
	\newunicodechar{γ}{\gamma}
	\newunicodechar{δ}{\delta}
	\newunicodechar{ε}{\varepsilon}
	\newunicodechar{η}{\eta}
	\newunicodechar{ι}{\iota}
	\newunicodechar{κ}{\kappa}
	\newunicodechar{λ}{\lambda}
	\newunicodechar{ρ}{\rho}
	\newunicodechar{μ}{\mu}
	\newunicodechar{ν}{\nu}
	\newunicodechar{π}{\pi}
	\newunicodechar{φ}{\varphi}
	\newunicodechar{ψ}{\psi}
	\newunicodechar{θ}{\theta}
	\newunicodechar{σ}{\sigma}
	\newunicodechar{τ}{\tau}
	\newunicodechar{χ}{\chi}
	\newunicodechar{ξ}{\xi}
	\newunicodechar{υ}{\upsilon}
	\newunicodechar{ζ}{\zeta}
	\newunicodechar{ω}{\omega}
	\newunicodechar{Γ}{\Gamma}
	\newunicodechar{Δ}{\Delta}
	\newunicodechar{Λ}{\Lambda}
	\newunicodechar{Φ}{\Phi}
	\newunicodechar{Ψ}{\Psi}
	\newunicodechar{Ω}{\Omega}
\else
	\setmathfont[range = {bb,"02190,"02192,"021F7,"021F8,"0290C,"0290D,"0219C,"0219D,"029C4,"029C5,"025C7}]{STIXTwoMath-Regular.otf}
\fi

%% file: setup/theoremstyle.tex
\usepackage{amsthm}
	\numberwithin{equation}{section}

\newtheoremstyle
	{plainer} 
	{2ex} 
	{2ex} 
	{} 
	{} 
	{\bfseries} 
	{} 
	{\newline} 
	{} 

\theoremstyle{plainer} 
\newtheorem{theorem}{Theorem}
\newtheorem{lemma}[theorem]{Lemma}
\newtheorem{corollary}[theorem]{Corollary}
\newtheorem{proposition}[theorem]{Proposition}

\newtheorem{definition}[theorem]{Definition}

\newtheorem{remark}[theorem]{Remark}

\numberwithin{theorem}{section}
\numberwithin{example}{section}

%% file: setup/codestyle.tex
\usepackage{xcolor} 
\usepackage{listings} 

\lstset
{
	tabsize = 2 ,
	basicstyle = \ttfamily ,
	showstringspaces = false ,
	breaklines=true ,
	breakatwhitespace=true ,
	language = ,
}

\lstset
{
	rangebeginprefix = \{--\ begin:\ ,
	rangebeginsuffix = \ --\},
	rangeendprefix = \{--\ end:\ ,
	rangeendsuffix = \ --\},
	includerangemarker = false
}

%% file: setup/tikzstyle.tex
\usepackage{tikz}

\usetikzlibrary
{
	calc ,
	matrix ,
	arrows ,
	angles ,
	quotes ,
	intersections ,
	positioning ,
	fit ,
	graphs ,
	shapes.geometric ,
	shapes.misc ,
	decorations ,
	decorations.markings ,
	decorations.pathmorphing ,
	decorations.pathreplacing ,
	decorations.text ,
	spath3
}

\ifluatex
	\usetikzlibrary{graphdrawing}
	\usegdlibrary{trees , force}
\else
	\relax
\fi

\usepgflibrary
{
	arrows.meta
}

\usepackage{xcolor}

\tikzset
{
	diagram/.style =
	{
		line cap = butt ,
		line join = bevel ,
		arrows = -> ,
		> = angle 60 ,
		auto = left ,
		font = \small ,
		text height = 1.5ex , 
		text depth = 0.25ex , 
		inner sep = 0.25em ,
	} ,
	smaller/.append style =
	{
		font = \small
	} ,
	code node/.append style =
	{
		every node/.append style =
		{
			execute at begin node = \begin{texttt} ,
			execute at end node = \end{texttt}
		}
	} ,
	code diagram/.style =
	{
		diagram ,
		code node
	} ,
	math node/.append style =
	{
		every node/.append style =
		{
			execute at begin node = \begin{math} ,
			execute at end node = \end{math}
		}
	} ,
	math diagram/.style =
	{
		diagram ,
		math node
	} ,
	string diagram/.style =
	{
		math diagram ,
		arrows = -
	} ,
	online/.style =
	{
		shape = rectangle ,
		rounded corners = 2mm ,
		inner sep = 1.5pt ,
		fill = white ,
		opacity = 1.0 ,
		anchor = center	
	} ,
	double arrow/.style =
	{
		double distance between line centers = 2pt ,
		-implies 
	} ,
	equals/.style =
	{
		double distance between line centers = 2pt ,
		-implies , 
		arrows = -
	} ,
	mapto/.append style =
	{
		arrows = |-> ,
	} ,
	paph/.append style =
	{
		decorate ,
		decoration = {snake , segment length = 5mm , amplitude = 0.5mm} ,
	} ,
	includel/.append style =
	{
		arrows = Hooks[right]-> ,
	} ,
	includer/.append style =
	{
		arrows = Hooks[left]-> ,
	} ,
	monomorphism/.append style =
	{
		arrows = >-> ,
	} ,
	epimorphism/.append style =
	{
		arrows = ->> ,
	} ,
	cofibration/.append style =
	{
		arrows = >-> ,
	} ,
	fibration/.append style =
	{
		arrows = ->> ,
	} ,
	equivalence/.append style =
	{
		decoration = {markings , mark = at position 1/2 with {\node[online , transform shape] {∼};}} ,
		postaction = {decorate} ,
	} ,
	mark pos/.store in = \markpos ,
	mark pos = 1/2 ,
	proarrow/.append style =
	{
		postaction = {decorate} ,
		decoration =
		{
			markings ,
			mark = at position \markpos with {\draw [solid , -] (0 , -0.6ex) to (0 , 0.6ex);}
		} ,
	} ,
	proarrows/.append style =
	{
		postaction = {decorate} ,
		decoration =
		{
			markings ,
			mark = at position #1with {\draw [solid , -] (0 , -0.6ex) to (0 , 0.6ex);}
		} ,
	} ,
	heteromorphism/.append style =
	{
		decoration = {markings , mark = at position \markpos with {\draw [white , solid , semithick , -] (-1.0ex , 0) to (1.0ex , 0);}} ,
		postaction = {decorate} ,
	} ,
	arrow/.append style =
	{
		arrows = -> ,
	} ,
	string/.append style =
	{
		arrows = - ,
	} ,
	cofibration string/.append style =
	{
		decoration = {markings , mark = at position 1/4 with {\arrow{>}} , mark = at position 3/4 with {\arrow{>}}} ,
		postaction = {decorate} ,
	} ,
	fibration string/.append style =
	{
		decoration = {markings , mark = at position 4/8 with {\arrow{>}} , mark = at position 5/8 with {\arrow{>}}} ,
		postaction = {decorate} ,
	} ,
	overcross/.append style =
	{
		preaction = {draw = white , - , line width = 4pt , line cap = round}
	} ,
	label/.append style =
	{
		font = \scriptsize
	} ,
	incoming/.append style =
	{
		pos = 0 ,
		anchor = south
	} ,
	outgoing/.append style =
	{
		pos = 1 ,
		anchor = north
	} ,
	outcoming/.append style =
	{
		pos = 0 ,
		anchor = north
	} ,
	ingoing/.append style =
	{
		pos = 1 ,
		anchor = south
	} ,
	fromabove/.append style =
	{
		pos = 0 ,
		anchor = south
	} ,
	tobelow/.append style =
	{
		pos = 1 ,
		anchor = north
	} ,
	frombelow/.append style =
	{
		pos = 0 ,
		anchor = north
	} ,
	toabove/.append style =
	{
		pos = 1 ,
		anchor = south
	} ,
	fromleft/.append style =
	{
		pos = 0 ,
		anchor = east
	} ,
	toright/.append style =
	{
		pos = 1 ,
		anchor = west
	} ,
	fromright/.append style =
	{
		pos = 0 ,
		anchor = west
	} ,
	toleft/.append style =
	{
		pos = 1 ,
		anchor = east
	} ,
	forall/.append style =
	{
		line width = 0.5pt ,
	} ,
	exists/.append style =
	{
		densely dashed
	} ,
	structural/.append style =
	{
		opacity = 2/3 ,
	} ,
	0-cell/.style =
	{
		shape = rectangle ,
		rounded corners = 2mm ,
		inner sep = 2pt
	} ,
	1-cell/.style =
	{
		shape = rectangle ,
		rounded corners = 2mm ,
		inner sep = 1.5pt ,
		fill = white ,
		opacity = 1.0 ,
		anchor = center	
	} ,
	2-cell/.style =
	{
		draw = black!75 ,
		fill = white ,
		opacity = 0.9 ,
		inner sep = 0.75mm ,
		minimum size = 4.0mm
	} ,
	bead/.style =
	{
		2-cell ,
		shape = rectangle ,
		rounded corners = 2.0mm
	} ,
	diamond/.style =
	{
		2-cell ,
		shape = diamond
	} ,
	bra/.style =
	{
		2-cell ,
		shape = isosceles triangle ,
		isosceles triangle apex angle = 70 ,
		shape border rotate = -90 ,
		minimum size = 5mm
	} ,
	ket/.style =
	{
		2-cell ,
		shape = isosceles triangle ,
		isosceles triangle apex angle = 70 ,
		shape border rotate = 90 ,
		minimum size = 6mm
	} ,
	box/.style =
	{
		2-cell ,
		shape = rectangle
	} ,
	dot/.style =
	{
		2-cell ,
		shape = circle ,
		text height = 0ex ,
		text depth = 0ex ,
		inner sep = 0mm ,
		minimum size = 1.5mm
	} ,
	black dot/.style =
	{
		dot ,
		fill = gray
	} ,
	white dot/.style =
	{
		dot ,
		fill = white
	} ,
	gray dot/.style =
	{
		dot ,
		fill = gray!25
	} ,
	sheet/.style =
	{
		draw = black!50 ,
		fill = gray!10 ,
		opacity = 1.0 ,
		fill opacity = 0.5
	} ,
	torn/.append style =
	{
		decoration = {random steps , segment length = 1.25pt , amplitude = 0.75pt}
	} ,
	on sheet/.append style =
	{
		semithick
	} ,
	edge annotation/.append style =
	{
		anchor = west
	} ,
	compositor/.style =
	{
		draw = black!50 ,
		fill = gray!10 ,
		draw opacity = 0.5 ,
		fill opacity = 0.25 ,
		dashed
	} ,
	callout/.append style =
	{
		minimum height = 2mm ,
		minimum width = 4mm ,
		draw ,
		draw opacity = 0.25
	} ,
	callout-pre/.append style =
	{
		callout ,
		inner sep = 1.2mm ,
		solid
	} ,
	callout-post/.append style =
	{
		callout ,
		inner sep = 0.8mm ,
		dashed
	} ,
	pullback/.pic =
	{
		\draw [semithick , arrows = -] (-0.1 , 0.1) -- (0.1 , 0.1) -- (0.1 , -0.1) ;
	} ,
	pushout/.pic =
	{
		\draw [semithick , arrows = -] (-0.1 , 0.1) -- (-0.1 , -0.1) -- (0.1 , -0.1) ;
	} ,
}

\pgfdeclaredecoration{simple line}{start}
{
  \state{start}[width = +0pt,
                next state=step]{
    \pgfpathmoveto{\pgfpoint{0pt}{0pt}}
  }
  \state{step}[auto end on length    = 3pt,
               auto corner on length = 3pt,               
               width=+1pt]
  {
    \pgfpathlineto{\pgfpoint{1pt}{0pt}}
  }
  \state{final}
  {}
}

%% file: setup/bibstyle.tex
\usepackage
[
	backend = biber ,
	style = alphabetic ,
	maxnames = 6 ,
	url = true ,
]{biblatex} 

\ExecuteBibliographyOptions
{
	sorting = nyt ,
	backref = false ,
}
\bibliography{bibliography} 

%% file: setup/hyperstyle.tex
\usepackage{hyperref}
\hypersetup
{
	unicode = true ,  
	bookmarks = true , 
	bookmarksopen = true , 
	bookmarksopenlevel = 2 , 
	breaklinks = true , 
	colorlinks = false , 
	pdfborder = {0 0 0} , 
	pdfborderstyle={/S/U/W 0} , 
}

%% file: macros/macros.tex
\NewDocumentCommand \define {s o m}
{%
	\hypertarget%
	{\IfValueTF{#2}{#2}{#3}}%
	{\IfBooleanTF{#1}{#3}{\emph{#3}}}%
}

\NewDocumentCommand \refer {s o m}
{%
	\hyperlink%
	{\IfValueTF {#2}{#2}{#3}}%
	{\IfBooleanTF{#1}{#3}{#3}}%
}

\NewDocumentCommand \liststartstheorem {} {\leavevmode \vspace{-1.0 \baselineskip}}  

\NewDocumentCommand \ensuretext {m} {\textrm{#1}}

\NewDocumentCommand \centerdisplay {m} {\makebox[\linewidth][c]{$\displaystyle{#1}$}}


\NewDocumentCommand \bbb {m} {\mathbb{#1}}    
\NewDocumentCommand \scp {m} {\ensuretext{\textsc{#1}}}  


\RenewDocumentCommand \arg {o}
{
	\IfNoValueTF{#1}
	{{−}}
	{\overset{#1}{−}}
}

\NewDocumentCommand \infarg {o}
{
	\IfNoValueTF{#1}
	{{\_}}
	{\overset{#1}{\_}}
}



\NewDocumentCommand \bind {o m m} 
{
	\IfNoValueTF{#1}
	{} {\mathop{{#1}}}
	{#2} \mathrel{.} \allowbreak {#3}
}

\NewDocumentCommand \unaryminus {} {\scalebox{0.5}[1.0]{\( - \)}}
\NewDocumentCommand \inv {m} {{#1} ^{\unaryminus1}} 

\NewDocumentCommand \oppositizer {} {∘}
\NewDocumentCommand \prooppositizer {} {∙}
\NewDocumentCommand \dualizer {} {∗}

\NewDocumentCommand \cat {m} {\bbb{#1}} 
\NewDocumentCommand \Cat {m} {\scp{#1}} 
\NewDocumentCommand \op {m} {{#1} ^{\oppositizer}} 
\NewDocumentCommand \co {m} {{#1} ^{\prooppositizer}} 
\NewDocumentCommand \dual {m} {{#1} ^{\dualizer}}
\NewDocumentCommand \arr {m} {{{#1} ^{\arrowizer}}} 

\makeatletter
\newcommand*{\length}[1]{%
    \@tempcnta\z@
    \@for\@tempa:=#1\do{\advance\@tempcnta\@ne}%
    \the\@tempcnta%
}
\makeatother

\makeatletter
\newcount\my@repeat@count
\newcommand{\natrec}[2]{%
  \begingroup
  \my@repeat@count=\z@
  \@whilenum\my@repeat@count<#1\do{#2\advance\my@repeat@count\@ne}%
  \endgroup
}
\makeatother

\DeclarePairedDelimiter \hombrackets {(} {)}
\NewDocumentCommand \makehom {d<> m} 
{
	\IfEqCase{#1}
	{%
		{dinatural}%
		{%
			\ifpdftex
				\def \arr {\di{#2}} 
			\else
				\def \arr {\ensurestackMath{\ThisStyle{\stackinset{c}{-0.2ex}{c}{0.4ex}{\SavedStyle{¨}}{\SavedStyle{#2}}}}}
			\fi
		}%
		{extranatural}%
		{%
			\def \arr {\ensurestackMath{\ThisStyle{\stackinset{c}{-0.2ex}{c}{1.0ex}{\SavedStyle{\scriptscriptstyle{∪\,|\, ∩}}}{\SavedStyle{#2}}}}}
		}%
	}%
	[\def \arr {#2}]
	\DeclareDocumentCommand \f {o m m} 
	{
		\IfNoValueTF{##1}
		{{##2} \mathrel{\arr} {##3}}
		{{##1} \, \allowbreak \hombrackets{{##2} \mathrel{\arr} {##3}}}
	}
}

\NewDocumentCommand \monoid {m m} 
{
	\DeclareDocumentCommand \f {m} 
	{
		\ifthenelse{\isempty{##1}}
		{
			{#2}
		}
		{
			\foreach \item [count=\index] in {##1}
			{
				\ifthenelse{\equal{\index}{1}}
				{} 
				{#1} 
				{\item}
			}
		}
	}
}

\DeclarePairedDelimiter \setbrackets {\{} {\}}
\NewDocumentCommand \set {m}
{
	\ifthenelse{\isempty{#1}}
	{∅}
	{
		\monoid{\mathpunct{,}}{\,}
		\setbrackets{\f{#1}}
	}
}

\DeclarePairedDelimiter \sequencebrackets {[} {]}
\NewDocumentCommand \sequence {m}
{
	\ifthenelse{\isempty{#1}}
	{∅}
	{
		\monoid{\mathpunct{,}}{\,}
		\sequencebrackets{\f{#1}}
	}
}


\NewDocumentCommand \arrow {}{→}
\RenewDocumentCommand \hom {d<> o m m} 
{
	\makehom <#1> {\arrow}
	\f[#2]{#3}{#4}
}

\NewDocumentCommand \comp {O{1} m}
{
	\monoid
	{\mathbin{\natrec{#1}{\mathord{⋅}}}}
	{\mathrm{id} \ifthenelse{#1 > 1}{^{#1}}{}}
	\f{#2}
}

\NewDocumentCommand \idty {o m}
{
	\comp[#1]{} ({#2})
}

\NewDocumentCommand \pmoc {O{1} m}
{
	\monoid
	{\mathbin{\natrec{#1}{\mathord{∘}}}}
	{\mathrm{id} \ifthenelse{#1 > 1}{^{#1}}{}}
	\f{#2}
}

\NewDocumentCommand \id {m}
{
	\mathrm{id}
	\ifthenelse
	{\equal{#1}{}}
	{}
	{({#1})}
}



\NewDocumentCommand \tensor {m}
{
	\monoid{⊗}{\mathrm{I}}
	\f{#1}
}






\NewDocumentCommand \product {m}
{
	\monoid{×}{1}
	\f{#1}
}

\DeclarePairedDelimiter \tuplebrackets {⟨} {⟩}
\DeclarePairedDelimiter \globaltuplebrackets {(} {)}
\NewDocumentCommand \tuple {s m}
{
	\IfBooleanTF{#1}
	{
		\monoid{\mathbin{,}}{\,}
		\globaltuplebrackets{\f{#2}}
	}
	{
		\ifthenelse{\isempty{#2}}
		{!}
		{
			\monoid{\mathbin{,}}{\,}
			\tuplebrackets{\f{#2}}
		}
	}
}

\NewDocumentCommand \coproduct {m}
{
	\monoid{+}{0}
	\f{#1}
}

\DeclarePairedDelimiter \cotuplebrackets {[} {]}
\NewDocumentCommand \cotuple {m}
{
	\ifthenelse{\isempty{#1}}
	{¡}
	{
		\monoid{\mathbin{,}}{\,}
		\cotuplebrackets{\f{#1}}
	}
}


\NewDocumentCommand \boundary {o}
{
	∂
	{
		\IfNoValueTF{#1}
		{}
		{^{#1}}
	}
}





\NewDocumentCommand \rightclass {m m} 
{
	\mathbin{\mathord{#2} \mkern-4mu {}^{\mathord{#1}}}
}
\NewDocumentCommand \leftclass {m m} 
{
	\mathbin{{}^{\mathord{#1}} \mkern-4mu {\mathord{#2}}}
}

\NewDocumentCommand \since {m} {\ensuremath{[\text{#1}]}}
\NewDocumentEnvironment{relationalreasoning}{}{\begin{array}{cl}}{\end{array}}
\NewDocumentCommand \term {o o m}
{
	\IfNoValueTF{#2}
	{{#1} & {#3}\\}
	{{#1} & \since{#2} \\ & {#3} \\}
}


\NewDocumentEnvironment{maptable}{}{\begin{array}{rcl}}{\end{array}}


\NewDocumentCommand \sequent {o o m m}
{
	{#3} \mathrel{⊢\IfNoValueTF{#1}{}{^{#1}}\IfNoValueTF{#2}{}{_{#2}}}{#4}
}


\NewDocumentCommand \cancel {o m} 
{
	\IfNoValueTF{#1}
	{[{#2}]}
	{{[{#2}] ^{#1}}}
}

%% file: macros/local_macros.tex
\NewDocumentCommand \di {m} {\ddot{#1}} 


\NewDocumentCommand \dummy {m} 
{
	\bar{#1}
}


\NewDocumentCommand \strip {s}
{
	\IfBooleanTF{#1}
	{\, | ^∗}
	{\, | \,}
}


\RenewDocumentCommand \qedsymbol {} 
{
	\ensuremath{□}
}



\NewDocumentCommand \braiding {} {σ}

\NewDocumentCommand \braid {o m m}
{
	\braiding
	{\IfNoValueTF{#1}{}{#1}}
	\tuple*{{#2} , {#3}}
}


\NewDocumentCommand \syllepsis {} {υ}

\NewDocumentCommand \syllepsize {o m m}
{
	\syllepsis
	{\IfNoValueTF{#1}{}{#1}}
	\tuple*{{#2} , {#3}}
}

\NewDocumentCommand \comparitor {o}
{
	\IfNoValueTF{#1}
	{θ}
	{θ_{#1}}
}



\RenewDocumentCommand \lim {} {\underset{⟵}{\mathrm {lim}}}


\NewDocumentCommand \leftmate {m}
{
	{}^{μ}{#1} %
}
\NewDocumentCommand \rightmate {m}
{
	{#1}{}^{μ} %
}

\NewDocumentCommand \push {o m m} 
{
	\IfNoValueTF{#1}
	{
		{#2}_!
	}
	{
		Σ_{#1} ({#2})
	}
	\ifthenelse{\isempty{#3}}{}{({#3})}
}
\NewDocumentCommand \forward {o m} 
{
	\overrightarrow{#2}
}
\NewDocumentCommand \pull {o m m} 
{
	\IfNoValueTF{#1}
	{
		{#2}^*
	}
	{
		\inv{#1} ({#2})
	}
	\ifthenelse{\isempty{#3}}{}{({#3})}
}
\NewDocumentCommand \backward {o m} 
{
	\overleftarrow{#2}
}

\NewDocumentCommand \cartesianlift {o m m} 
{
	\overset{\hom{∙}{#3}}{#2}
}
\NewDocumentCommand \opcartesianlift {o m m} 
{
	\overset{\hom{#3}{∙}}{#2}
}



\NewDocumentCommand \proarrow {}{⇸}
\NewDocumentCommand \prohom {d<> o m m} 
{
	\makehom <#1> {\proarrow}
	\f[#2]{#3}{#4}
}

\NewDocumentCommand \procomp {m}
{
	\monoid{⊙}{\mathrm{U}}
	\f{#1}
}












\NewDocumentCommand \het {o m m}
{
	\makehom{⤍}
	\f[#1]{#2}{#3}
}


\NewDocumentCommand \pathhom {o m m}
{
	\makehom{⇝}
	\f[#1]{#2}{#3}
}







\NewDocumentCommand \lifts {}
{⧄}
\NewDocumentCommand \orthogonal {s m m}
{
	\IfBooleanTF{#1}
		{{#2} \mathrel{\overline{\lifts}} {#3}}
		{{#2} \mathrel{\lifts} {#3}}
}

\NewDocumentCommand \homotopy {o m m}
{
	\IfNoValueTF{#1}
	{
		\makehom{↝}
		\f{#2}{#3}
	}
	{
		\makehom{\overset{#1} {↝}}
		\f{#2}{#3}
	}
}

\NewDocumentCommand \hcomp {O{1} m}
{
	\monoid
	{\mathbin{\natrec{#1}{\mathord{+}}}}
	{0 \ifthenelse{#1 > 1}{^{#1}}{}}
	\f{#2}
}






\NewDocumentCommand \doublehom {o m m m m} 
{
	\IfNoValueTF{#1}
	{\prescript{#4\vphantom{#3}}{#2\vphantom{#5}}{}{◇}{}^{#3\vphantom{#4}}_{#5\vphantom{#2}}}
	{{#1} \, \allowbreak \hombrackets{\prescript{#4\vphantom{#3}}{#2\vphantom{#5}}{}{◇}{}^{#3\vphantom{#4}}_{#5\vphantom{#2}}}}
}




\NewDocumentCommand \fix {o m}
{
	\mathop{\mathrm{fix} {\IfValueTF {#1}{_{#1}}{}}} \ifthenelse{\isempty{#2}}{}{({#2})}
}




\NewDocumentCommand \boundrel {o m m}
{
	{#2} \mathbin{⪯} {#3}
	\IfNoValueTF{#1}{}{ : #1}
}


\NewDocumentCommand \costaction{o m m} 
{
	{#2} \mathbin{+_{\IfNoValueTF{#1}{⋅}{#1}}} {#3}
}



\NewDocumentCommand \evaluatesto {o m m}  
{
	{#2} \mathrel{\mathord{⇓}{\IfNoValueTF{#1}{}{^{#1}}}} {#3}
}


%% file: content/outline.tex

\begin{abstract} \pdfbookmark[1]{Abstract}{section: abstract}
	\subimport*{parts/}{abstract}

\end{abstract}

\keywords
{
	double categories,
	dinaturality,
	surface diagrams
}

\msc
{
	18N10, 
	18A23, 
	18M30 
}

\begin{sect}{Introduction} \label{section: introduction}
	\subimport*{parts/}{introduction}

\end{sect}

\begin{sect}{Double Categorical Preliminaries} \label{section: double categories}
	\subimport*{parts/}{double_categories}

\end{sect}

\begin{sect}{Dinaturality} \label{section: dinaturality}
	\subimport*{parts/}{dinaturality}

\end{sect}

\begin{sect}{Diagrammatics of Dinaturality} \label{section: diagrammatics}
	\subimport*{parts/}{diagrammatics}
\end{sect}

\begin{sect}{Extranaturality from Dinaturality} \label{section: extranaturality}
	\subimport*{parts/}{extranaturality}

\end{sect}

\begin{sect}{Conclusion and Prospectus} \label{section: conclusion}
	\subimport*{parts/}{conclusion}

\end{sect}

\begin{sect*}{References} \pdfbookmark[1]{References}{section: references}
	\printbibliography[heading=none]
\end{sect*}

\appendix

\newpage

\begin{sect}{Proofs} \label{section: proofs}
	\subimport*{parts/}{proofs}

\end{sect}


%% file: content/parts/abstract.tex

In this paper we extend the concept of dinaturality to the setting of double categories.
We introduce the dinatural versions of double categorical transformations and modifications,
and show that ordinary natural transformations and modifications
correspond to dinatural ones between dummy functors.
We prove the double categorical generalizations
of some classical theorems about dinatural transformations for $1$-categories,
and extend the surface diagram calculus
for the locally cubical Gray category of small double categories
to include dinatural constructions.
In our motivating example of dinatural constructions
we reconstruct the double categorical formulation of mates in terms of dinatural transformations
and characterize extranaturality for $1$-categories using dinaturality for double categories.

%% file: content/parts/introduction.tex

At least proverbially, category theory was invented to explain what it means for a family of maps to be \emph{natural}.
The structure of a \emph{natural transformation} of functors was a first attempt to address this question
\cite{mac_lane-1945-naturality}.
But it was immediately clear that this structure captures only some of the ways
in which a family of maps can be intuitively considered natural.
For example, the familiar \emph{evaluation map} in the category of sets and functions,
taking an ordered pair of sets $\tuple*{A , B}$
to the function $\mathrm{eval} \, \tuple*{A , B} : \hom{\product{A , (\hom{A}{B})}}{B}$,
is intuitively natural, in the sense of being generic in both $A$ and $B$,
yet it is not a natural transformation.
While it is formally natural in $B$, in that given a function $b : \hom{B}{B′}$
we have $\comp{\product{A , (\hom{A}{b})} , \mathrm{eval} \, \tuple*{A , B′}} = \comp{\mathrm{eval} \, \tuple*{A , B} , b}$,
this reasoning does not work for $A$.
Roughly speaking, naturality links the $B$ in the domain of $\mathrm{eval} \, \tuple*{A , B}$ with the $B$ in the codomain.
But there is no $A$ in the codomain, while there are two in the domain, one appearing covariantly and the other contravariantly;
and $\mathrm{eval}$ somehow links these with each other.

It was in the context of studying such maps of mixed-variance functors
that Eilenberg and Kelly proposed expanding the concept of naturality
in a way that includes both \emph{ordinary naturality}, like that of $B$ above,
and \emph{extraordinary naturality}, like that of $A$
\cite{eilenberg-1966-functorial_calculus}.
Dubuc and Street introduced the structure of a \emph{dinatural transformation} \cite{street-1970-dinatural_transformations},
and used it to decompose Eilenberg and Kelly's \emph{generalized natural transformations} (a.k.a. ``extranatural transformations'')
into constituent parts.
A dinatural transformation is a map between parallel functors that can each take as input both covariant and contravariant data.
Roughly speaking, in a dinatural transformation data can be linked
across the two boundary functors with a consistent variance,
or within one of them with complementary variance \cite{mccusker-2021-dinatuality}.

Here, we extend the concept of dinatural transformation to the setting of double categories.
Because double categories have two independent dimensions, we must respect the variance of constructions in each of them.
Because they have an additional dimension of structure relative to $1$-categories,
we expect the original coherence relation for dinatural transformations
to correspond to a new family of $2$-cells,
which will in turn require its own coherence relations.
Because double categories and their hierarchy of morphisms have three dimensions of structure,
we can define $3$-cell \emph{dimodifications} that mediate between dinatural transformations in an appropriate way.
Dinaturality for $1$-categories provides a common setting for the study of both ordinary and extraordinary naturality,
and this is true in the case of double categories as well.

The plan of the paper is as follows.
In section \ref{section: double categories}
we review some basic constructions in double categories that we will need later
and introduce two lemmas that will help mediate between natural and dinatural structure.
In section \ref{section: dinaturality}
we introduce the dinatural extensions of natural transformations and modifications for double categories.
We prove the double categorical generalizations of some classical results about dinatural transformations for $1$-categories,
and describe the algebra of dinatural--natural composition in the double categorical setting.
In section \ref{section: diagrammatics}
we take a detour into the diagrammatics of dinaturality,
and show how the $2$-dimensional algebra of dinatural transformations and dimodifications for double categories
arises by projection from a simpler $3$-dimensional algebra.
Section \ref{section: extranaturality}
constitutes an extended example of the application of dinaturality for double categories.
We use dinatural constructions in a monoidal double category
to recover the double categorical presentation of mates and to define an extranatural square,
such that an extranatural square in the double category of small categories is an extranatural transformation.

%% file: content/parts/double_categories.tex

We begin by briefly reviewing some basic concepts of double category theory used in the sequel.
Further details can be found in, for example, \cite{grandis-2019-multiple_categories}.
We will work in the setting of \define[weak double category]{weak double categories},
where composition is strict in the \define{arrow dimension} $\hom{}{}$,
with composition ``$\comp{,}$'' and units ``$\comp{}$'',
but associative and unital only up to coherent isomorphism in the \define{proarrow dimension} $\prohom{}{}$,
with composition ``$\procomp{,}$'' and units ``$\procomp{}$''.
This is largely because we wish to accomodate examples in the double category of (suitably small) categories, $\Cat{Cat}$,
where proarrows are conventionally taken to be profunctors without coherently chosen representations.
However, this choice matters little due to a coherence theorem ensuring that any weak double category is suitably equivalent
to one with strict composition in both dimensions \cite{grandis-1999-limits_in_double_categories}.

Despite working in the weak setting, we will omit proarrow composition coherator isomorphisms from diagrams,
and, with warning, from algebraic expressions as well.
The main reason for this is that ternary composition plays a central role in dinatural constructions.
We can,  of course, assign arbitrary binary bracketings to ternary composites
and use associator isomorphisms to manipulate the bracketing as needed.
But this adds considerable notational overhead that can distract from the conceptual core of some constructions.
For equations between squares in a weak double category
that rely on suppressed composition coherators to unify the boundaries we write ``$≅$'' rather than ``$=$''.
This convention is intended to strike a balance between explicitness and legibility.

Our terminology and notation for double categories follow \cite{morehouse-2023-cartesian_gray_monoidal_double_categories},
except that we represent the \refer{arrow dimension} horizontally and the \refer{proarrow dimension} vertically.
This way vertical and horizontal composites of natural transformations in $\Cat{Cat}$ get depicted according to their names.
Composition is written in left-to-right order.
The boundary specification for a \define[double category square]{square} is written as
``$\doublehom[\doublehom[\cat{C}]{\prohom{A}{C}}{\prohom{B}{D}}{\hom{A}{B}}{\hom{C}{D}}]{m}{n}{f}{g}$'', where
$f$ and $g$ are \define[double category arrow]{arrows},
$m$ and $n$ are \define[double category poroarrow]{proarrows}, and
$A$, $B$, $C$, and $D$ are \define[double category object]{objects} of the double category $\cat{C}$.
Any well-formed prefix of a boundary specification may be omitted if it is clear from context or irrelevant.
Because double categories are $2$-dimensional structures
it is often useful to describe composable configurations of their squares
using string diagrams \cite{myers-2016-string_diagrams_for_double_categories}.
The string diagram for a square $φ$ with the boundary described above is shown below on the left.
We will often omit inferable or unreferenced cell labels for diagram legibility.
$$

$$

We refer to squares having only identity boundary morphisms in some dimension
as \define[globular square]{globular} or \define[disk]{disks},
and use the abbreviations ``$\prohom{f}{g}$'' for $\doublehom{\procomp{}}{\procomp{}}{f}{g}$
and ``$\hom{m}{n}$'' for $\doublehom{m}{n}{\comp{}}{\comp{}}$.
String diagrams for disks $σ$ and $τ$ with these boundaries are shown above in the middle.
An \define{invertible disk} is one with a two-sided compositional inverse in the dimension where it has not-necessarily-trivial boundary.
An \define{identity disk} is one of the form $\procomp{} \, f : \prohom{f}{f}$ or $\comp{} \, m : \hom{m}{m}$.
A (double) \define{identity square} $\comp[2]{} \, A : \doublehom{\procomp{}}{\procomp{}}{\comp{}}{\comp{}}$
is an identity disk in both dimensions.
String diagrams for these identity disks are shown above on the right.
Note the convention to suppress the drawing of identity cells.
Likewise, in algebraic expressions we often elide the ``$\comp{}$'' or ``$\procomp{}$''
for an identity cell whose dimension is evident from context.

Given a double category $\cat{C}$,
its \define{proarrow-dimension reflection} $\co{\cat{C}}$ (pronounced ``$\cat{C}$-co'')
is the double category with the same objects, arrows, proarrows, and squares,
but with cell boundaries in the proarrow dimension transposed
and composition order in that dimension reversed.
Thus, each square
$φ : \doublehom[\doublehom[\cat{C}]{\prohom{A}{C}}{\prohom{B}{D}}{\hom{A}{B}}{\hom{C}{D}}]{m}{n}{f}{g}$
has a doppelgänger
$φ : \doublehom[\doublehom[\co{\cat{C}}]{\prohom{C}{A}}{\prohom{D}{B}}{\hom{C}{D}}{\hom{A}{B}}]{m}{n}{g}{f}$,
and string diagrams in $\co{\cat{C}}$ look like those in $\cat{C}$ ``flipped upside down''.
For example, the following diagrams depict a composite square $\procomp{φ , ψ}$ in $\cat{C}$ and in $\co{\cat{C}}$.
$$

	\quad : \quad
	\co{\cat{C}}
$$

The \define{empty double category} $\cat{\coproduct{}}$ has no objects.
A (chosen) \define{singleton double category} $\cat{\product{}}$
has just one object, one arrow, one proarrow, and one square.
Given double categories $\cat{C}$ and $\cat{D}$,
the \define{cartesian ordered pair double category}
$\product{\cat{C} , \cat{D}}$
has objects, arrows, proarrows, and squares
consisting of ordered pairs of the same sort of cell from $\cat{C}$ and $\cat{D}$ respectively,
with the composition structure given factor-wise.
As usual, $\product{,}$ and $\cat{\product{}}$ determine
finite \define{cartesian product} structure for categories of double categories.
 
A (strict) \define{functor} of double categories consists of boundary-preserving functions of objects, arrows, proarrows, and squares
that preserve the composition structure in both dimensions.
A functor whose domain is a cartesian product is called \define[dummy functor]{dummy} in one of its arguments
if it factors through the projection to the other.
For a functor $F : \hom{\cat{C}_i}{\cat{D}}$ we write ``$\dummy{F}$''
for the dummy functor $\comp{π_i , F} : \hom{\product{\cat{C}_0 , \cat{C}_1}}{\cat{D}}$.

A pair of \refer[functor]{functors}
$\tensor{,} : \hom{\product{\cat{C} , \cat{C}}}{\cat{C}}$ and $\tensor{} : \hom{\cat{\product{}}}{\cat{C}}$
that obey the monoid laws endow $\cat{C}$ with the structure of a \define{monoidal double category}.
Given squares $φ$ and $ψ$ of a monoidal double cateory,
we can ``deproject'' the string diagram for the square $\tensor{φ , ψ}$
shown on the left to the surface diagram shown on the right,
using surface juxtaposition to represent the monoidal product.
$$

$$
By \define{monoidal naturality} this square is equal (at least up to coherators) to each of the squares listed below.
Diagrammatically, sub-diagrams embedded in different surfaces are independent,
and may be perturbed within their respective boundaries
without altering the value of the composite square.
$$
	\comp{(\tensor{φ , p}) , (\tensor{n , ψ})}
	\quad , \quad
	\comp{(\tensor{m , ψ}) , (\tensor{φ , q})}
	\quad , \quad
	\procomp{(\tensor{φ , i}) , (\tensor{g , ψ})}
	\quad , \quad
	\procomp{(\tensor{f , ψ}) , (\tensor{φ , j})}
$$

The higher-dimensional structure of double categories and their functors
has (natural) transformations as $2$-cells and modifications as $3$-cells.
Because these play a central role in the sequel, we recall their definitions and composition structure in detail.
Natural transformations for double categories come in two dimensions,
corresponding to the \refer[arrow dimension]{arrow} and \refer[proarrow dimension]{proarrow} dimensions
of double categories themselves.

\begin{definition}[proarrow-dimension natural transformation]
	For parallel functors of double categories
	$F , G : \hom{\cat{C}}{\cat{D}}$
	a (pseudo) \define{proarrow-dimension natural transformation}
	$α : \prohom{F}{G}$
	consists of the following data.
	\begin{description}
		\item[{\define*[proarrow-dimension natural transformation object-component proarrow]{object-component proarrows}}:]
			for each $\cat{C}$-object $A$
			a $\cat{D}$-proarrow
			$α A : \prohom{F A}{G A}$,
		\item[{\define*[proarrow-dimension natural transformation arrow-component square]{arrow-component squares}}:]
			for each $\cat{C}$-arrow $f : \hom{A}{B}$
			a $\cat{D}$-square $α f : \doublehom{α A}{α B}{F f}{G f}$,
			$$

			$$
		\item[{\define*[proarrow-dimension natural transformation proarrow-component disk]{proarrow-component disks}}:]
			for each $\cat{C}$-proarrow $m : \prohom{A}{B}$
			an \refer[invertible disk]{invertible} $\cat{D}$-proarrow \refer{disk} $α m : \hom{\procomp{F m , α B}}{\procomp{α A , G m}}$.
			$$
				%
			$$
	\end{description}
	This data is required to satisfy the following relations.
	\begin{description}
		\item[preservation of arrow composition:]
			for an object $A$
			and consecutive arrows $f : \hom{A}{B}$ and $g : \hom{B}{C}$
			of $\cat{C}$
			we have
			\begin{equation} \label{proarrow-dimension natural transformation arrow composition preservation}
				α (\comp{} \, A) = \comp{} \, α A
				\quad \text{and} \quad
				α (\comp{f , g}) = \comp{α f , α g}
			\end{equation}
			$$
				%
			$$
		\item[compatibility with proarrow composition:]
			for an object $A$
			and consecutive proarrows $m : \prohom{A}{B}$ and $n : \prohom{B}{C}$
			of $\cat{C}$
			we have
			\begin{equation} \label{proarrow-dimension natural transformation proarrow composition compatibility}
				α (\procomp{} \, A) ≅ \comp{} \, α A
				\quad \text{and} \quad
				α (\procomp{m , n}) ≅ \comp{(\procomp{F m , α n}) , ({\procomp{α m , G n}})}
			\end{equation}
			$$
				%
			$$
			where the boundaries agree up to coherators and
			an internal associator has been omitted from both the last expression and its diagram.
		\item[naturality for squares:]
			for a square $φ : \doublehom[\doublehom[\cat{C}]{\prohom{A}{C}}{\prohom{B}{D}}{\hom{A}{B}}{\hom{C}{D}}]{m}{n}{f}{g}$
			we have
			\begin{equation} \label{proarrow-dimension natural transformation square naturality}
				\comp{(\procomp{F φ , α g}) , α n} = \comp{α m , (\procomp{α f , G φ})}
			\end{equation}
			$$
				%
			$$
	\end{description}
\end{definition}

\begin{remark}[oriented and strict natural transformations]
	In the preceding definition we have stipulated that the
	\refer[proarrow-dimension natural transformation proarrow-component disk]{proarrow-component disks}
	be \refer[invertible disk]{invertible}.
	If we drop this requirement then we obtain oriented natural transformations.
	In particular, with proarrow-component disks oriented in the direction that we have defined them
	we obtain an \define{oplax natural transformation},
	and for the reverse direction a \define{lax natural transformation}.
	Alternatively, we can strengthen this stipulation to require that the proarrow-component disks be \refer[identity disk]{identities},
	in which case we obtain a \define{strict natural transformation}.
\end{remark}

\begin{remark}[arrow-dimension natural transformations] \label{arrow-dimension natural transformations}
	Diagrammatically, an \define{arrow-dimension natural transformation} $\hom{F}{G}$
	is just the reflection of a \refer{proarrow-dimension natural transformation} $\prohom{F}{G}$
	about the $\tuple*{\comp{,} , \procomp{,}}$-diagonal.
	However, because of the differing strictness of composition structure in the two dimensions,
	the laws are not exactly algebraically dual once the coherators elided above are explicitly inserted.
\end{remark}

\paragraph{Natural transformation composition}
Natural transformations compose component-wise,
in the sense that for a functor $F$
we have an \define{identity natural transformation} $\procomp{} \, F : \prohom{F}{F}$
with components
$(\procomp{} \, F) \, A = \procomp{} \, F A$,
$(\procomp{} \, F) \, f = \procomp{} \, F f$, and
$(\procomp{} \, F) \, m ≅ \comp{} \, F m$;
$$

$$
and for consecutive natural transformations $α : \prohom{F}{G}$ and $β : \prohom{G}{H}$
we have a \define{composite natural transformation} $\procomp{α , β} : \prohom{F}{H}$
with components
$(\procomp{α , β}) A = \procomp{α A , β A}$,
$(\procomp{α , β}) f = \procomp{α f , β f}$, and
$(\procomp{α , β}) m ≅ \comp{(\procomp{α m , β B}) , (\procomp{α A , β m})}$.
$$
	%
$$

Compatible pairs of transformations in each dimension together determine square-shaped boundaries.
Cells inhabiting these boundaries are known as modifications.

\begin{definition}[modification] \label{definition: modification}
	For parallel functors of double categories
	$F , G , I , J : \hom{\cat{C}}{\cat{D}}$,
	\refer[arrow-dimension natural transformation]{arrow-dimension natural transformations}
	$α : \hom{F}{G}$ and $β : \hom{I}{J}$,
	and \refer[proarrow-dimension natural transformation]{proarrow-dimension natural transformations}
	$γ : \prohom{F}{I}$ and $δ : \prohom{G}{J}$,
	a (cubical) \define{modification}
	$
		μ : \doublehom[\doublehom[\hom[]{\cat{C}}{\cat{D}}]
		{\prohom{F}{I}}{\prohom{G}{J}}{\hom{F}{G}}{\hom{I}{J}}]
		{γ}{δ}{α}{β}
	$
	consists of the following data.
	\begin{description}
		\item[{\define*[modification object-component square]{object-component squares}}:]
			for each $\cat{C}$-object $A$
			a $\cat{D}$-square $μ A : \doublehom{γ A}{δ A}{α A}{β A}$.
			$$
				\begin{tikzpicture}[string diagram , x = {(12mm , 0mm)} , y = {(0mm , -12mm)} , baseline=(align.base)]
					\coordinate (cell) at (1/2 , 1/2) ;
					\draw [name path = back arr]
						(0 , 0 |- cell) to node [fromleft] {γ A} node [toright] {δ A}
						(1 , 1 |- cell) ;
					\draw [name path = back pro]
						(cell |- 0 , 0) to node [fromabove] {α A} node [tobelow] {β A}
						(cell |- 1 , 1) ;
					\path [name intersections = {of = back arr and back pro , by = {square}}] ;
					\node at ($ (square) + (-1/3 , -1/3) $) {F A} ;
					\node at ($ (square) + (1/3 , -1/3) $) {G A} ;
					\node at ($ (square) + (-1/3 , 1/3) $) {I A} ;
					\node at ($ (square) + (1/3 , 1/3) $) {J A} ;
					\node [bead] (align) at (square) {μ A} ;
				\end{tikzpicture}
			$$
	\end{description}
	This data is required to satisfy the following relations.
	\begin{description}
		\item[naturality for arrows:]
			for an arrow $f : \hom[\cat{C}]{A}{B}$
			we have
			\begin{equation} \label{modification arrow naturality}
				\procomp{α f , (\comp{μ A , δ f})}  ≅  \procomp{(\comp{γ f , μ B}) , β f}
			\end{equation}
			$$
				\begin{tikzpicture}[string diagram , x = {(18mm , 0mm)} , y = {(0mm , -18mm)} , baseline=(align.base)]
					\coordinate (cell) at (1/2 , 4/5) ;
					\coordinate (left top) at (1/5 , 0) ;
					\coordinate (right bot) at (4/5 , 1) ;
					\draw [name path = back pro]
						(0 , 0 |- cell) to [out = east , in = west] node [fromleft] {γ A} node [toright] {δ B}
						(1 , 1 |- cell) ;
					\draw [name path =  back arr]
						(right bot |- left top) to [out = south , in = north] node [fromabove] {α B}
						(left top |- cell) to [out = south , in = north] node [tobelow] {β A}
						(left top |- right bot) ;
					\draw [name path = front arr]
						(left top) to [out = south , in = north] node [fromabove] {F f}
						(right bot |- cell) to [out = south , in = north] node [tobelow] {J f}
						(right bot) ;
					\path [name intersections = {of = back arr and back pro , by = {functor square}}] ;
					\path [name intersections = {of = back pro and front arr , by = {interchanger square}}] ;
					\path [name intersections = {of = back arr and front arr , by = {interchanger disk}}] ;
					\node [bead] at (functor square) {μ A} ;
					\node [bead] at (interchanger square) {δ f} ;
					\node [bead] at (interchanger disk) {α f} ;
					\node (align) at (1/2 , 1/2) {} ;
				\end{tikzpicture}
				\, ≅ \,
				\begin{tikzpicture}[string diagram , x = {(18mm , 0mm)} , y = {(0mm , -18mm)} , baseline=(align.base)]
					\coordinate (cell) at (1/2 , 1/5) ;
					\coordinate (left top) at (1/5 , 0) ;
					\coordinate (right bot) at (4/5 , 1) ;
					\draw [name path = back pro]
						(0 , 0 |- cell) to [out = east , in = west] node [fromleft] {γ A} node [toright] {δ B}
						(1 , 1 |- cell) ;
					\draw [name path =  back arr]
						(right bot |- left top) to [out = south , in = north] node [fromabove] {α B}
						(right bot |- cell) to [out = south , in = north] node [tobelow] {β A}
						(left top |- right bot) ;
					\draw [name path = front arr]
						(left top) to [out = south , in = north] node [fromabove] {F f}
						(left top |- cell) to [out = south , in = north] node [tobelow] {J f}
						(right bot) ;
					\path [name intersections = {of = back arr and back pro , by = {functor square}}] ;
					\path [name intersections = {of = back pro and front arr , by = {interchanger square}}] ;
					\path [name intersections = {of = back arr and front arr , by = {interchanger disk}}] ;
					\node [bead] at (functor square) {μ B} ;
					\node [bead] at (interchanger square) {γ f} ;
					\node [bead] at (interchanger disk) {β f} ;
					\node (align) at (1/2 , 1/2) {} ;
				\end{tikzpicture}
			$$
		\item[naturality for proarrows:]
			for a proarrow $m : \prohom[\cat{C}]{A}{B}$
			we have
			\begin{equation} \label{modification proarrow naturality}
				\comp{γ m , (\procomp{μ A , β m})} = \comp{(\procomp{α m , μ B}) , δ m}
			\end{equation}
			$$
				\begin{tikzpicture}[string diagram , x = {(24mm , 0mm)} , y = {(0mm , -16mm)} , baseline=(align.base)]
					\coordinate (cell) at (4/5 , 1/2) ;
					\coordinate (left top) at (0 , 1/5) ;
					\coordinate (right bot) at (1 , 4/5) ;
					\draw [name path = back pro]
						(left top |- right bot) to [out = east , in = west] node [fromleft] {γ B}
						(cell |- left top) to [out = east , in = west] node [toright] {δ A}
						(right bot |- left top) ;
					\draw [name path =  back arr]
						(cell |- 0 , 0) to [out = south , in = north] node [fromabove] {α A} node [tobelow] {β B}
						(cell |- 1 , 1) ;
					\draw [name path = front pro]
						(left top) to [out = east , in = west] node [fromleft] {F m}
						(cell |- right bot) to [out = east , in = west] node [toright] {J m}
						(right bot) ;
					\path [name intersections = {of = back arr and back pro , by = {functor square}}] ;
					\path [name intersections = {of = back arr and front pro , by = {interchanger square}}] ;
					\path [name intersections = {of = back pro and front pro , by = {interchanger disk}}] ;
					\node [bead] at (functor square) {μ A} ;
					\node [bead] at (interchanger square) {β m} ;
					\node [bead] at (interchanger disk) {γ m} ;
					\node (align) at (1/2 , 1/2) {} ;
				\end{tikzpicture}
				\, = \,
				\begin{tikzpicture}[string diagram , x = {(24mm , 0mm)} , y = {(0mm , -16mm)} , baseline=(align.base)]
					\coordinate (cell) at (1/5 , 1/2) ;
					\coordinate (left top) at (0 , 1/5) ;
					\coordinate (right bot) at (1 , 4/5) ;
					\draw [name path = back pro]
						(left top |- right bot) to [out = east , in = west] node [fromleft] {γ B}
						(cell |- right bot) to [out = east , in = west] node [toright] {δ A}
						(right bot |- left top) ;
					\draw [name path =  back arr]
						(cell |- 0 , 0) to [out = south , in = north] node [fromabove] {α A} node [tobelow] {β B}
						(cell |- 1 , 1) ;
					\draw [name path = front pro]
						(left top) to [out = east , in = west] node [fromleft] {F m}
						(cell |- left top) to [out = east , in = west] node [toright] {J m}
						(right bot) ;
					\path [name intersections = {of = back arr and back pro , by = {functor square}}] ;
					\path [name intersections = {of = back arr and front pro , by = {interchanger square}}] ;
					\path [name intersections = {of = back pro and front pro , by = {interchanger disk}}] ;
					\node [bead] at (functor square) {μ B} ;
					\node [bead] at (interchanger square) {α m} ;
					\node [bead] at (interchanger disk) {δ m} ;
					\node (align) at (1/2 , 1/2) {} ;
				\end{tikzpicture}
			$$
	\end{description}
\end{definition}

A modification is \define[globular modification]{globular}
if its object-component squares are \refer[globular square]{globular}.
A globular modification is \define[invertible modification]{invertible}
if its object-component disks are \refer[invertible disk]{invertible}.

\paragraph{Modification composition}
Modifications also compose component-wise,
in the sense that for natural transformations $α : \hom{F}{G}$ and $γ : \prohom{F}{I}$
we have \define[identity modification]{identity modifications} $\procomp{} \, α : \prohom{α}{α}$ and $\comp{} \, γ : \hom{γ}{γ}$
with components
$(\procomp{} \, α) \, A = \procomp{} \, α A$ and
$(\comp{} \, γ) \, A = \comp{} \, γ A$;
$$

$$
and for modifications $μ$, $σ$, and $τ$,
with $μ$ and $σ$ consecutive in the \refer{proarrow dimension}
and $μ$ and $τ$ consecutive in the \refer{arrow dimension},
we have \define[composite modification]{composite modifications} $\procomp{μ , σ}$ and $\comp{μ , τ}$
with components
$(\procomp{μ , σ}) A = \procomp{μ A , σ A}$ and
$(\comp{μ , τ}) A = \comp{μ A , τ A}$.
$$
	%
$$

The following two lemmas will be our main tools to mediate between natural and dinatural structure for double categories.
Because they can be rather bulky,
proofs of the propositions in this paper are collected in appendix \ref{section: proofs}.
In the electronic version the tombstone ending a proposition is a link to its proof
and the proof number is a link back to the proposition.

\begin{lemma}[natural transformation product Yang-Baxter] \label{theorem: natural transformation product yang baxter}
	For a \refer{proarrow-dimension natural transformation} $α : \prohom[\hom[]{\product{\cat{B} , \cat{C}}}{\cat{D}}]{F}{G}$,
	and proarrows $p : \prohom[\cat{B}]{X}{Y}$ and $m : \prohom[\cat{C}]{A}{B}$,
	we have the following equation.
	\begin{equation} \label{natural transformation product yang baxter}

		\tag*{\hyperlink{proof: natural transformation product yang baxter}{\qedsymbol}}
	\]
\end{lemma}

\begin{lemma}[dummy natural transformations and modifications] \label{theorem: dummy natural transformations}
	Each \refer[proarrow-dimension natural transformation]{natural transformation}
	$α : \prohom[\hom[]{\cat{C}}{\cat{D}}]{F}{G}$
	determines a natural transformation of \refer[dummy functor]{dummy functors}
	$\dummy{α} : \prohom[\hom[]{\product{\cat{B} , \cat{C}}}{\cat{D}}]{\dummy{F}}{\dummy{G}}$,
	which we call a \define{dummy natural transformation}, with
	\refer[proarrow-dimension natural transformation object-component proarrow]{object-component proarrows}
	$\dummy{α} \tuple*{X , A} ≔ α A$,
	\refer[proarrow-dimension natural transformation arrow-component square]{arrow-component squares}
	$\dummy{α} \tuple*{i , f} ≔ α f$, and
	\refer[proarrow-dimension natural transformation proarrow-component disk]{proarrow-component disks}
	$\dummy{α} \tuple*{p , m} ≔ α m$.
	Likewise, each \refer{modification}
	$
		μ : \doublehom[\doublehom[\hom[]{\cat{C}}{\cat{D}}]
		{\prohom{F}{I}}{\prohom{G}{J}}{\hom{F}{G}}{\hom{I}{J}}]
		{γ}{δ}{α}{β}
	$
	determines a  \define{dummy modification}
	$
		\dummy{μ} : \doublehom[\doublehom[\hom[]{\product{\cat{B} , \cat{C}}}{\cat{D}}]
		{\prohom{\dummy{F}}{\dummy{I}}}{\prohom{\dummy{G}}{\dummy{J}}}{\hom{\dummy{F}}{\dummy{G}}}{\hom{\dummy{I}}{\dummy{J}}}]
		{\dummy{γ}}{\dummy{δ}}{\dummy{α}}{\dummy{β}}
	$
	with
	\refer[modification object-component square]{object-component squares}
	$\dummy{μ} \tuple*{X , A} ≔ μ A$.
	\hfill \hyperlink{proof: dummy natural transformations}{\qedsymbol}
\end{lemma}

%% file: content/parts/dinaturality.tex

In the setting of $1$-categories,
Dubuc and Street observed that there is another natural notion of transformation between parallel functors,
which they called a ``dinatural transformation'' \cite{street-1970-dinatural_transformations}.
Whereas a natural transformation of mixed-variance $1$-functors $\hom[\hom[]{\product{\op{\cat{C}} , \cat{C}}}{\cat{D}}]{F}{G}$
has component arrows indexed by an ordered pair of $\cat{C}$-objects,
a dinatural transformation $\hom<dinatural>[\hom[]{\product{\op{\cat{C}} , \cat{C}}}{\cat{D}}]{F}{G}$
-- note the diaeresis -- has component arrows indexed by a single $\cat{C}$-object,
and a coherence relation extended to ensure harmony between the two variances.
The concept of dinatural transformation was extended to $2$-categories by Vidal and Tur \cite{vidal-2010-2_categorical_institution}.
Here we extend it to double categories
and define a dinatural version of modification as well.

The starting point for dinatural constructions is a pair of categorical structures related by a dimensional reflection.

\begin{definition}[dual pairing]
	The (proarrow-dimension) \define{dual pairing} for a double category $\cat{C}$
	is the \refer{cartesian ordered pair double category}
	that it forms with its \refer{proarrow-dimension reflection},
	$\product{\co{\cat{C}} , \cat{C}}$.
\end{definition}

More generally, we refer to the product of any categorical structure with any of its dimensional reflections as a dual pairing.
For example, we have dual pairings of $1$-categories under the \emph{opposite category} reflection, $\product{\op{\cat{C}} , \cat{C}}$.

\begin{definition}[difunctor]
	We refer to a functor whose domain is a dual pairing as a \define{difunctor}.
\end{definition}

In particular, for each functor of double categories $F : \hom{\cat{C}}{\cat{D}}$
there is a \define{dummy difunctor} $\dummy{F} ≔ \comp{π_1 , F} : \hom{\product{\co{\cat{C}} , \cat{C}}}{\cat{D}}$.

\begin{proposition}[dummy difunctors are functors] \label{theorem: dummy difunctors are functors}
	The \refer{dummy functor} construction determines a bijection
	between arbitrary functors of double categories
	and dummy difunctors.
	\hfill \hyperlink{proof: dummy difunctors are functors}{\qedsymbol}
\end{proposition}

We arrive at the dinatural version of a \refer[proarrow-dimension natural transformation]{natural transformation} for double categories.

\begin{definition}[proarrow-dimension dinatural transformation]
	For parallel difunctors of double categories
	$F , G : \hom{\product{\co{\cat{C}} , \cat{C}}}{\cat{D}}$
	a (pseudo) \define{proarrow-dimension dinatural transformation}
	$α : \prohom<dinatural>{F}{G}$
	consists of the following data.
	\begin{description}
		\item[{\define*[proarrow-dimension dinatural transformation object-component proarrow]{object-component proarrows}}:]
			for each $\cat{C}$-object $A$
			a $\cat{D}$-proarrow
			$α A : \prohom{F \tuple*{A , A}}{G \tuple*{A , A}}$,
		\item[{\define*[proarrow-dimension dinatural transformation arrow-component square]{arrow-component squares}}:]
			for each $\cat{C}$-arrow $f : \hom{A}{B}$
			a $\cat{D}$-square $α f : \doublehom{α A}{α B}{F \tuple*{f , f}}{G \tuple*{f , f}}$,
			$$

			$$
		\item[{\define*[proarrow-dimension dinatural transformation proarrow-component disk]{proarrow-component disks}}:]
			for each $\cat{C}$-proarrow $m : \prohom{A}{B}$
			an \refer[invertible disk]{invertible} $\cat{D}$-proarrow \refer{disk}
			$α m : \hom{\procomp{F \tuple*{B , m} , α B , G \tuple*{m , B}}}{\procomp{F \tuple*{m , A} , α A , G \tuple*{A , m}}}$,
			$$
				%
			$$
			where we have suppressed arbitrary binary bracketings of the ternary boundaries.
	\end{description}

	This data is required to satisfy the following relations.
	\begin{description}
		\item[preservation of arrow composition:]
			for an object $A$
			and consecutive arrows $f : \hom{A}{B}$ and $g : \hom{B}{C}$
			of $\cat{C}$
			we have
			\begin{equation} \label{proarrow-dimension dinatural transformation arrow composition preservation}
				α (\comp{} \, A) = \comp{} \, α A
				\quad \text{and} \quad
				α (\comp{f , g}) = \comp{α f , α g}
			\end{equation}
			$$
				%
			$$
			where the boundaries agree by the functoriality of $F$ and $G$
			and the fact that composition in a \refer{cartesian ordered pair double category} is defined factor-wise. 
		\item[compatibility with proarrow composition:]
			for an object $A$
			and consecutive proarrows $m : \prohom{A}{B}$ and $n : \prohom{B}{C}$
			of $\cat{C}$
			we have
			\begin{equation} \label{proarrow-dimension dinatural transformation proarrow composition compatibility}
				%
			$$
			where the boundaries agree up to coherators by the functoriality of $F$ and $G$
			and the fact that composition in a \refer{cartesian ordered pair double category} is defined factor-wise.
			Internal associators have been omitted from both the last expression and its diagram.
		\item[naturality for squares:]
			for a square $φ : \doublehom[\doublehom[\cat{C}]{\prohom{A}{C}}{\prohom{B}{D}}{\hom{A}{B}}{\hom{C}{D}}]{m}{n}{f}{g}$
			we have
			\begin{equation} \label{proarrow-dimension dinatural transformation square naturality}
				\comp{(\procomp{F \tuple*{g , φ} , α g , G \tuple*{φ , g}}) , α n}
				\; = \;
				\comp{α m , (\procomp{F \tuple*{φ , f} , α f , G \tuple*{f , φ}})}
			\end{equation}
			$$
				\begin{tikzpicture}[string diagram , x = {(30mm , 0mm)} , y = {(0mm , -20mm)} , baseline=(align.base)]
					\coordinate (cell) at (1/4 , 1/2) ;
					\coordinate (left top) at (0 , 1/6) ;
					\coordinate (right bot) at (1 , 5/6) ;
					\draw [name path = back pro]
						(left top |- cell) to [out = east , in = west] node [fromleft] {α C} node [toright] {α B}
						(right bot |- cell) ;
					\draw [name path = front falling pro]
						(left top) to [out = east , in = west] node [fromleft] {F \tuple*{C , m}}
						(cell |- left top) to [out = east , in = west] node [toright] {G \tuple*{B , n}}
						(right bot) ;
					\draw [name path = front rising pro]
						(left top |- right bot) to [out = east , in = west] node [fromleft] {G \tuple*{m , C}}
						(cell |- right bot) to [out = east , in = west] node [toright] {F \tuple*{n , B}}
						(right bot |- left top) ;
					\draw [name path = front arr]
						(cell |- 0 , 0) to [out = south , in = north] node [fromabove] {F \tuple*{g , f}} node [tobelow] {G \tuple*{f , g}}
						(cell |- 1 , 1) ;
					\path [name intersections = {of = front arr and front falling pro , by = {falling functor square}}] ;
					\path [name intersections = {of = front arr and front rising pro , by = {rising functor square}}] ;
					\path [name intersections = {of = front arr and back pro , by = {interchanger square}}] ;
					\path [name intersections = {of = front falling pro and front rising pro , by = {interchanger disk}}] ;
					\node [bead] at (falling functor square) {F \tuple*{g , φ}} ;
					\node [bead] at (rising functor square) {G \tuple*{φ , g}} ;
					\node [bead] at (interchanger square) {α g} ;
					\node [bead] (align) at (interchanger disk) {α n} ;
				\end{tikzpicture}
				\, = \,
				\begin{tikzpicture}[string diagram , x = {(30mm , 0mm)} , y = {(0mm , -20mm)} , baseline=(align.base)]
					\coordinate (cell) at (3/4 , 1/2) ;
					\coordinate (left top) at (0 , 1/6) ;
					\coordinate (right bot) at (1 , 5/6) ;
					\draw [name path = back pro]
						(left top |- cell) to [out = east , in = west] node [fromleft] {α C} node [toright] {α B}
						(right bot |- cell) ;
					\draw [name path = front falling pro]
						(left top) to [out = east , in = west] node [fromleft] {F \tuple*{C , m}}
						(cell |- right bot) to [out = east , in = west] node [toright] {G \tuple*{B , n}}
						(right bot) ;
					\draw [name path = front rising pro]
						(left top |- right bot) to [out = east , in = west] node [fromleft] {G \tuple*{m , C}}
						(cell |- left top) to [out = east , in = west] node [toright] {F \tuple*{n , B}}
						(right bot |- left top) ;
					\draw [name path = front arr]
						(cell |- 0 , 0) to [out = south , in = north] node [fromabove] {F \tuple*{g , f}} node [tobelow] {G \tuple*{f , g}}
						(cell |- 1 , 1) ;
					\path [name intersections = {of = front arr and front falling pro , by = {falling functor square}}] ;
					\path [name intersections = {of = front arr and front rising pro , by = {rising functor square}}] ;
					\path [name intersections = {of = front arr and back pro , by = {interchanger square}}] ;
					\path [name intersections = {of = front falling pro and front rising pro , by = {interchanger disk}}] ;
					\node [bead] at (falling functor square) {G \tuple*{f , φ}} ;
					\node [bead] at (rising functor square) {F \tuple*{φ , f}} ;
					\node [bead] at (interchanger square) {α f} ;
					\node [bead] (align) at (interchanger disk) {α m} ;
				\end{tikzpicture}
			$$
			assuming a consistent proarrow bracketing on the two sides of the equation.
	\end{description}
\end{definition}

Note how for the \refer[proarrow-dimension dinatural transformation proarrow-component disk]{proarrow-component disks}
we factor the proarrow $\tuple*{m , m} : \prohom[(\product{\co{\cat{C}},  \cat{C}})]{\tuple*{B , A}}{\tuple*{A , B}}$
to compose first the covariant $m$ and then the contravariant $m$ in one boundary,
and the other way around in the other boundary,
leaving in between equal objects to serve as a single index to the transformation.

To support the claim that this definition of dinatural transformation for double categories is reasonable,
we observe that it reduces to Vidal and Tur's $2$-categorical definition \cite{vidal-2010-2_categorical_institution}
for double categories in which all squares are \refer[globular square]{globular}\footnote
{modulo composition strictness and the apparent omission of composition compatibility laws like
 \eqref{proarrow-dimension dinatural transformation proarrow composition compatibility}},
and to Dubuc and Street's $1$-categorical definition \cite{street-1970-dinatural_transformations}
when, additionally, all \refer[disk]{disks} are \refer[identity disk]{identities}.

\begin{remark}[oriented and strict dinatural transformations]
	In the preceding definition we have stipulated that
	\refer[proarrow-dimension dinatural transformation proarrow-component disk]{proarrow-component disks}
	be \refer[invertible disk]{invertible}.
	If we drop this requirement then we obtain oriented dinatural transformations.
	In particular, with proarrow-component disks oriented in the direction that we have defined them
	we obtain an \define{oplax dinatural transformation},
	and for the reverse direction a \define{lax dinatural transformation}.
	Alternatively, we can strengthen this stipulation to require that the proarrow-component disks be \refer[identity disk]{identities},
	in which case we obtain a \define{strict dinatural transformation}.
\end{remark}

\begin{remark}[arrow-dimension dinatural transformations]
	One might suppose that an \define{arrow-dimension dinatural transformation}
	$\hom<dinatural>{F}{G}$
	should be (modulo composition strictness) the reflection of its proarrow-dimension counterpart
	about the $\tuple*{\comp{,} , \procomp{,}}$-diagonal, as in the natural case (remark \ref{arrow-dimension natural transformations}).
	However, such a definition makes sense only if the diagonal reflection is applied to the \refer{dual pairing} as well,
	giving $\hom<dinatural>[\hom[]{\product{\op{\cat{C}} , \cat{C}}}{\cat{D}}]{F}{G}$,
	where $\op{\cat{C}}$ is the \define{arrow-dimension reflection} of $\cat{C}$.
\end{remark}

Comparing the structure of a dinatural transformation with that of a natural transformation of difunctors,
we see that the former is defined only ``on the diagonal'',
in the sense that components have a single index shared by the dual pairing,
whereas components of the latter have two independent indices.
As in the $1$-categorical setting \cite{street-1970-dinatural_transformations},
by unifying these indices we can construct a dinatural transformation from a natural one.

\begin{proposition}[natural transformation diagonalization] \label{theorem: transformation diagonalization}
	For each \refer[proarrow-dimension natural transformation]{natural transformation} of difunctors
	$α : \prohom[\hom[]{\product{\co{\cat{C}} , \cat{C}}}{\cat{D}}]{F}{G}$
	there is a \refer[proarrow-dimension dinatural transformation]{dinatural transformation}
	$\di{α} : \prohom<dinatural>[\hom[]{\product{\co{\cat{C}} , \cat{C}}}{\cat{D}}]{F}{G}$,
	which we call its \define[transformation diagonalization]{diagonalization}, with
	\refer[proarrow-dimension dinatural transformation object-component proarrow]{object-component proarrows} $\di{α} A ≔ α \tuple*{A , A}$,
	\refer[proarrow-dimension dinatural transformation arrow-component square]{arrow-component squares} $\di{α} f ≔ α \tuple*{f , f}$,
	and for $m : \prohom{A}{B}$,
	\refer[proarrow-dimension dinatural transformation proarrow-component disk]{proarrow-component disks}
	\[
		\centerdisplay{

		\overset{\eqref{natural transformation product yang baxter}}{≅}
		%
		\tag*{\hyperlink{proof: transformation diagonalization}{\qedsymbol}}
		}
	\]
\end{proposition}

\begin{corollary}[dummy dinatural transformations] \label{theorem: dummy dinatural transformations}
	For each \refer[proarrow-dimension natural transformation]{natural transformation}
	$α : \prohom[\hom[]{\cat{C}}{\cat{D}}]{F}{G}$
	there is a \refer[proarrow-dimension dinatural transformation]{dinatural transformation} of \refer[dummy difunctor]{dummy difunctors}
	$\di{\dummy{α}} : \prohom<dinatural>[\hom[]{\product{\co{\cat{C}} , \cat{C}}}{\cat{D}}]{\dummy{F}}{\dummy{G}}$,
	which we call a \define{dummy dinatural transformation}, with
	\refer[proarrow-dimension dinatural transformation object-component proarrow]{object-component proarrows} $\di{\dummy{α}} A ≔ α A$,
	\refer[proarrow-dimension dinatural transformation arrow-component square]{arrow-component squares} $\di{\dummy{α}} f ≔ α f$, and
	\refer[proarrow-dimension dinatural transformation proarrow-component disk]{proarrow-component disks} $\di{\dummy{α}} m :≅ α m$.
	\[
		\centerdisplay{

		\tag*{\hyperlink{proof: dummy dinatural transformations}{\qedsymbol}}
	}
	\]
\end{corollary}

In general, we cannot invert the operation of diagonalizing a natural transformation of difunctors to a dinatural transformation
because it loses the information of what happens ``off the diagonal''.
But we can invert the creation of a dummy dinatural transformation
because requiring the difunctors to be dummy in the same factor
reduces the laws for a dinatural transformation to those for a natural transformation.
Thus, as in the $1$-categorical setting \cite{street-1970-dinatural_transformations},
a dummy dinatural transformation ``is'' just a natural transformation.

\begin{proposition} [dummy dinatural transformations are natural transformations] \label{theorem: dummy dinatural transformations are natural transformations}
	The \refer{dummy dinatural transformation} construction determines a bijection
	between arbitrary \refer[proarrow-dimension natural transformation]{natural transformations}
	and \refer[proarrow-dimension dinatural transformation]{dinatural transformations} of \refer[dummy difunctor]{dummy difunctors}:
	\[
		\prohom[\hom[]{\cat{C}}{\cat{D}}]{F}{G}
		\quad ≅ \quad
		\prohom<dinatural>[\hom[]{\product{\co{\cat{C}} , \cat{C}}}{\cat{D}}]{\dummy{F}}{\dummy{G}}
		\tag*{\hyperlink{proof: dummy dinatural transformations are natural transformations}{\qedsymbol}}
	\]
\end{proposition}

Diagonalizing an \refer{identity natural transformation}
$\procomp{} \, F : \prohom[\hom[]{\product{\co{\cat{C}} , \cat{C}}}{\cat{D}}]{F}{F}$
yields a dinatural transformation
$\di{\procomp{}} \, F : \prohom<dinatural>[\hom[]{\product{\co{\cat{C}} , \cat{C}}}{\cat{D}}]{F}{F}$
with
\refer[proarrow-dimension dinatural transformation object-component proarrow]{object components}
$(\di{\procomp{}} \, F) \, A = \procomp{} \, F \tuple*{A , A}$,
\refer[proarrow-dimension dinatural transformation arrow-component square]{arrow components}
$(\di{\procomp{}} \, F) \, f = \procomp{} \, F \tuple*{f , f}$,
and for $m : \prohom{A}{B}$,
\refer[proarrow-dimension dinatural transformation proarrow-component disk]{proarrow components}
$(\di{\procomp{}} \, F) \, m ≅ \comp{} \, F \tuple*{m , m}$.
$$
	\begin{tikzpicture}[string diagram , x = {(20mm , 0mm)} , y = {(0mm , -14mm)} , baseline=(align.base)]
		\coordinate (cell) at (1/2 , 1/2) ;
		\coordinate (left top) at (0 , 1/8) ;
		\coordinate (right bot) at (1 , 7/8) ;
		\draw [name path = back pro]
			(left top |- cell) to [out = east , in = west] node [fromleft] {(\di{\procomp{}} \, F) B} node [toright] {(\di{\procomp{}} \, F) A}
			(right bot |- cell) ;
		\draw [name path = front falling pro]
			(left top) to [out = east , looseness = 1.5 , in = west] node [fromleft] {F \tuple*{B , m}} node [toright] {F \tuple*{A , m}}
			(right bot) ;
		\draw [name path = front rising pro]
			(left top |- right bot) to [out = east , looseness = 1.5 , in = west] node [fromleft] {F \tuple*{m , B}} node [toright] {F \tuple*{m , A}}
			(right bot |- left top) ;
		\path [name intersections = {of = front falling pro and front rising pro , by = {square}}] ;
		\node [bead] (align) at (square) {(\di{\procomp{}} \, F) \, m} ;
	\end{tikzpicture}
	\; ≅ \;
	\begin{tikzpicture}[string diagram , x = {(16mm , 0mm)} , y = {(0mm , -14mm)} , baseline=(align.base)]
		\coordinate (cell) at (1/2 , 1/2) ;
		\draw [name path = back pro]
			(0 , 0 |- cell) to [out = east , in = west] node [fromleft] {F \tuple*{m , m}} coordinate [pos = 1/2] (square) node [toright] {F \tuple*{m , m}}
			(1 , 1 |- cell) ;
		\node (align) at (square) {} ;
	\end{tikzpicture}
$$

Diagonalizing a \refer{composite natural transformation}
$\procomp{α , β} : \prohom[\hom[]{\product{\co{\cat{C}} , \cat{C}}}{\cat{D}}]{F}{H}$
yields a dinatural transformation
$α \di{\procomp{,}} β : \prohom<dinatural>[\hom[]{\product{\co{\cat{C}} , \cat{C}}}{\cat{D}}]{F}{H}$
with
\refer[proarrow-dimension dinatural transformation object-component proarrow]{object components}
$(α \di{\procomp{,}} β) \, A = \procomp{α \tuple*{A , A} , β \tuple*{A , A}}$,
\refer[proarrow-dimension dinatural transformation arrow-component square]{arrow components}
$(α \di{\procomp{,}} β) \, f = \procomp{α \tuple*{f , f} , β \tuple*{f , f}}$,
and for $m : \prohom[\cat{C}]{A}{B}$,
\refer[proarrow-dimension dinatural transformation proarrow-component disk]{proarrow components}
$(α \di{\procomp{,}} β) \, m ≅ \comp{(\procomp{α \tuple*{B , m} , \inv{β \tuple*{m , B}}}) , (\procomp{\inv{α \tuple*{m , A}} , β \tuple*{A , m}})}$,
suppressing internal associators as usual.
$$

$$

Unlike natural transformations,
consecutive dinatural transformations are not necessarily composable.
The problem stems from the proarrow-component disks.
For parallel difunctors $F , G , H : \hom{\product{\co{\cat{C}} , \cat{C}}}{\cat{D}}$
and consecutive dinatural transformations $α : \prohom<dinatural>{F}{G}$ and $β : \prohom<dinatural>{G}{H}$,
in order to define a composite of $α$ and $β$, for a proarrow $m : \prohom[\cat{C}]{A}{B}$
we would need to construct a disk $(\procomp{α , β}) m$ from $α m$ and $β m$:
\[
	%
\]
As in the $1$-categorical setting \cite{laretto-2024-directed_type_theory},
we can compose dinatural transformations
if the domain double category is groupoidal in the corresponding dimension.

\begin{proposition} [groupoidal dinatural composition] \label{theorem: groupoidal dinatural composition}
	For parallel difunctors of double categories $F , G , H : \hom{\product{\co{\cat{C}} , \cat{C}}}{\cat{D}}$,
	if $\cat{C}$ has proarrow-dimension inverses for squares (and thus inverses for proarrows)
	then from consecutive \refer[proarrow-dimension dinatural transformation]{dinatural transformations}
	$α : \prohom<dinatural>{F}{G}$ and $β : \prohom<dinatural>{G}{H}$
	we construct a dinatural transformation $\procomp{α , β} : \prohom<dinatural>{F}{H}$,
	which we call their \define[dinatural composition]{composite}, with
	\refer[proarrow-dimension dinatural transformation object-component proarrow]{object-component proarrows}
	$(\procomp{α , β}) A ≔ \procomp{α A , β A}$,
	\refer[proarrow-dimension dinatural transformation arrow-component square]{arrow-component squares}
	$(\procomp{α , β}) f ≔ \procomp{α f , β f}$,
	and for $m : \prohom{A}{B}$,
	\refer[proarrow-dimension dinatural transformation proarrow-component disk]{proarrow-component disks}
	$$

	$$
	where the unlabeled bends $($ and $)$ are the identity disks
	\[
		\begin{array}{cccc}
			\procomp{} \, G \tuple*{B , B} ≅ \procomp{G \tuple*{m , B} , G \tuple*{\inv{m} , B}}
			& , &
			\procomp{G \tuple*{A , m} , G \tuple*{A , \inv{m}}} ≅  \procomp{} \, G \tuple*{A , A}
			& ,
			\\
			\procomp{} \, G \tuple*{B , B} ≅ \procomp{G \tuple*{B , \inv{m}} , G \tuple*{B , m}}
			& , &
			\procomp{G \tuple*{\inv{m} , A} , G \tuple*{m , A}} ≅  \procomp{} \, G \tuple*{A , A}
			& .
			\\
		\end{array}
		\tag*{\hyperlink{proof: groupoidal dinatural composition}{\qedsymbol}}
	\]
\end{proposition}

Also as in the $1$-categorical setting \cite{street-1970-dinatural_transformations},
we can define a composition operation for dinatural transformations with natural transformations.

\begin{proposition} [dinatural--natural composition] \label{theorem: dinatural-natural composition}
	From a \refer[proarrow-dimension dinatural transformation]{dinatural transformation} $α : \prohom<dinatural>{F}{G}$
	and a consecutive \refer[proarrow-dimension natural transformation]{natural transformation} $β : \prohom{G}{H}$
	we construct a dinatural transformation $\procomp{α , β} : \prohom<dinatural>{F}{H}$,
	which we call their \define[dinatural-natural composition]{composite}, with
	\refer[proarrow-dimension dinatural transformation object-component proarrow]{object-component proarrows}
	$(\procomp{α , β}) A ≔ \procomp{α A , β \tuple*{A , A}}$,
	\refer[proarrow-dimension dinatural transformation arrow-component square]{arrow-component squares}
	$(\procomp{α , β}) f ≔ \procomp{α f , β \tuple*{f , f}}$,
	and for $m : \prohom{A}{B}$,
	\refer[proarrow-dimension dinatural transformation proarrow-component disk]{proarrow-component disks}
	$
		(\procomp{α , β}) m ≅
		\comp
			{
				(\procomp{F \tuple*{B , m} , α B , \inv{β \tuple*{m , B}}}) ,
				(\procomp{α m , β \tuple*{A , B}}) ,
				(\procomp{F \tuple*{m , A} , α A , β \tuple*{A , m}})
			}
	$.
	\[

		\tag*{\hyperlink{proof: dinatural-natural composition}{\qedsymbol}}
	\]
\end{proposition}

Similarly, we have the dinatural composite $\procomp{α , β}$
of a natural $α : \prohom{F}{G}$ with a dinatural $β : \prohom<dinatural>{G}{H}$.
We can think of these mixed composition operations
as a two-sided \emph{action} of natural transformations on dinatural transformations.

\begin{proposition} [dinatural--natural composition laws] \label{theorem: dinatural-natural composition laws}
	\liststartstheorem
	\begin{itemize}
		\item
			\refer[identity natural transformation]{Identity natural transformations}
			are units for dinatural--natural composition:
			for dinatural $α : \prohom<dinatural>{F}{G}$
			we have
			$\procomp{\procomp{} \, F , α} ≅ α ≅ \procomp{α , \procomp{} \, G}$.
		\item
			Dinatural--natural composition with a \refer{composite natural transformation}
			is iterated dinatural--natural composition:
			for consecutive $α$, $β$, $γ$, $δ$, $ε$, with $γ$ dinatural and the rest natural
			we have
			$\procomp{(\procomp{α , β}) , γ} ≅ \procomp{α , (\procomp{β , γ})}$, and
			$\procomp{γ , (\procomp{δ , ε})} ≅ \procomp{(\procomp{γ , δ}) , ε}$.
		\item
			Two-sided dinatural--natural composition is associative:
			for consecutive $α$, $β$, $γ$, with $β$ dinatural and the rest natural
			we have
			$\procomp{(\procomp{α , β}) , γ} ≅ \procomp{α , (\procomp{β , γ})}$.
		\item
			Dinatural--natural composition with a diagonalized \refer{identity natural transformation}
			is \refer[transformation diagonalization]{diagonalization}:
			for natural $α : \prohom{F}{G}$
			we have
			$\procomp{\di{\procomp{}} \, F , α} ≅ \di{α} ≅ \procomp{α , \di{\procomp{}} \, G}$.
			\hfill \hyperlink{proof: dinatural-natural composition laws}{\qedsymbol}
	\end{itemize}
\end{proposition}

Once we have proarrow-dimension dinatural transformations of difunctors as a new kind of proarrow,
we can define a new kind of square having them in their boundaries.
These constitute the dinatural version of a \refer{modification} for double categories.

\begin{definition}[dimodification]
	For parallel difunctors of double categories
	$F , G , I , J : \hom{\product{\co{\cat{C}} , \cat{C}}}{\cat{D}}$,
	\refer[arrow-dimension natural transformation]{arrow-dimension natural transformations}
	$α : \hom{F}{G}$ and $β : \hom{I}{J}$,
	and \refer[proarrow-dimension dinatural transformation]{proarrow-dimension dinatural transformations}
	$γ : \prohom<dinatural>{F}{I}$ and $δ : \prohom<dinatural>{G}{J}$,
	a \define{dimodification}
	$
		μ : \doublehom[\doublehom[\hom[]{\product{\co{\cat{C}} , \cat{C}}}{\cat{D}}]
		{\prohom<dinatural>{F}{I}}{\prohom<dinatural>{G}{J}}{\hom{F}{G}}{\hom{I}{J}}]
		{γ}{δ}{α}{β}
	$
	consists of the following data.
	\begin{description}
		\item[{\define*[dimodification object-component square]{object-component squares}}:]
			for each $\cat{C}$-object $A$
			a $\cat{D}$-square $μ A : \doublehom{γ A}{δ A}{α \tuple*{A , A}}{β \tuple*{A , A}}$.
			$$
				\begin{tikzpicture}[string diagram , x = {(16mm , 0mm)} , y = {(0mm , -12mm)} , baseline=(align.base)]
					\coordinate (cell) at (1/2 , 1/2) ;
					\draw [name path = back arr]
						(0 , 0 |- cell) to node [fromleft] {γ A} node [toright] {δ A}
						(1 , 1 |- cell) ;
					\draw [name path = back pro]
						(cell |- 0 , 0) to node [fromabove] {α \tuple*{A , A}} node [tobelow] {β \tuple*{A , A}}
						(cell |- 1 , 1) ;
					\path [name intersections = {of = back arr and back pro , by = {square}}] ;
					\node at ($ (square) + (-1/2 , -1/3) $) {F \tuple*{A , A}} ;
					\node at ($ (square) + (1/2 , -1/3) $) {G \tuple*{A , A}} ;
					\node at ($ (square) + (-1/2 , 1/3) $) {I \tuple*{A , A}} ;
					\node at ($ (square) + (1/2 , 1/3) $) {J \tuple*{A , A}} ;
					\node [bead] (align) at (square) {μ A} ;
				\end{tikzpicture}
			$$
	\end{description}
	This data is required to satisfy the following relations.
	\begin{description}
		\item[naturality for arrows:]
			for an arrow $f : \hom[\cat{C}]{A}{B}$
			we have
			\begin{equation} \label{dimodification arrow naturality}
				\procomp{α \tuple*{f , f} , (\comp{μ A , δ f})}  ≅  \procomp{(\comp{γ f , μ B}) , β \tuple*{f , f}}
			\end{equation}
			$$
				\begin{tikzpicture}[string diagram , x = {(18mm , 0mm)} , y = {(0mm , -18mm)} , baseline=(align.base)]
					\coordinate (cell) at (1/2 , 4/5) ;
					\coordinate (left top) at (1/5 , 0) ;
					\coordinate (right bot) at (4/5 , 1) ;
					\draw [name path = back pro]
						(0 , 0 |- cell) to [out = east , in = west] node [fromleft] {γ A} node [toright] {δ B}
						(1 , 1 |- cell) ;
					\draw [name path =  back arr]
						(right bot |- left top) to [out = south , in = north] node [fromabove] {\quad α \tuple*{B , B}}
						(left top |- cell) to [out = south , in = north] node [tobelow] {β \tuple*{A , A} \quad}
						(left top |- right bot) ;
					\draw [name path = front arr]
						(left top) to [out = south , in = north] node [fromabove] {F \tuple*{f , f} \quad}
						(right bot |- cell) to [out = south , in = north] node [tobelow] {\quad J \tuple*{f , f}}
						(right bot) ;
					\path [name intersections = {of = back arr and back pro , by = {functor square}}] ;
					\path [name intersections = {of = back pro and front arr , by = {interchanger square}}] ;
					\path [name intersections = {of = back arr and front arr , by = {interchanger disk}}] ;
					\node [bead] at (functor square) {μ A} ;
					\node [bead] at (interchanger square) {δ f} ;
					\node [bead] at (interchanger disk) {α \tuple*{f , f}} ;
					\node (align) at (1/2 , 1/2) {} ;
				\end{tikzpicture}
				\, ≅ \,
				\begin{tikzpicture}[string diagram , x = {(18mm , 0mm)} , y = {(0mm , -18mm)} , baseline=(align.base)]
					\coordinate (cell) at (1/2 , 1/5) ;
					\coordinate (left top) at (1/5 , 0) ;
					\coordinate (right bot) at (4/5 , 1) ;
					\draw [name path = back pro]
						(0 , 0 |- cell) to [out = east , in = west] node [fromleft] {γ A} node [toright] {δ B}
						(1 , 1 |- cell) ;
					\draw [name path =  back arr]
						(right bot |- left top) to [out = south , in = north] node [fromabove] {\quad α \tuple*{B , B}}
						(right bot |- cell) to [out = south , in = north] node [tobelow] {β \tuple*{A , A} \quad}
						(left top |- right bot) ;
					\draw [name path = front arr]
						(left top) to [out = south , in = north] node [fromabove] {F \tuple*{f , f} \quad}
						(left top |- cell) to [out = south , in = north] node [tobelow] {\quad J \tuple*{f , f}}
						(right bot) ;
					\path [name intersections = {of = back arr and back pro , by = {functor square}}] ;
					\path [name intersections = {of = back pro and front arr , by = {interchanger square}}] ;
					\path [name intersections = {of = back arr and front arr , by = {interchanger disk}}] ;
					\node [bead] at (functor square) {μ B} ;
					\node [bead] at (interchanger square) {γ f} ;
					\node [bead] at (interchanger disk) {β \tuple*{f , f}} ;
					\node (align) at (1/2 , 1/2) {} ;
				\end{tikzpicture}
			$$
		\item[naturality for proarrows:]
			for a proarrow $m : \prohom[\cat{C}]{A}{B}$
			we have
			\begin{equation} \label{dimodification proarrow naturality}
				\comp{γ m , (\procomp{α \tuple*{m , A}, μ A , β \tuple*{A , m}})} = \comp{(\procomp{α \tuple*{B , m} , μ B , β \tuple*{m , B}}) , δ m}
			\end{equation}
			$$
				\begin{tikzpicture}[string diagram , x = {(30mm , 0mm)} , y = {(0mm , -20mm)} , baseline=(align.base)]
					\coordinate (cell) at (3/4 , 1/2) ;
					\coordinate (left top) at (0 , 1/6) ;
					\coordinate (right bot) at (1 , 5/6) ;
					\draw [name path =  back arr]
						(cell |- 0 , 0) to [out = south , in = north] node [fromabove] {α \tuple*{B , A}} node [tobelow] {β \tuple*{A , B}}
						(cell |- 1 , 1) ;
					\draw [name path = back pro]
						(left top |- cell) to [out = east , in = west] node [fromleft] {γ B} node [toright] {δ A}
						(right bot |- cell) ;
					\draw [name path = front falling pro]
						(left top) to [out = east , in = west] node [fromleft] {F \tuple*{B , m}}
						(cell |- right bot) to [out = east , in = west] node [toright] {J \tuple*{A , m}}
						(right bot) ;
					\draw [name path = front rising pro]
						(left top |- right bot) to [out = east , in = west] node [fromleft] {I \tuple*{m , B}}
						(cell |- left top) to [out = east , in = west] node [toright] {G \tuple*{m , A}}
						(right bot |- left top) ;
					\path [name intersections = {of = back arr and back pro , by = {functor square}}] ;
					\path [name intersections = {of = back arr and front falling pro , by = {falling interchanger square}}] ;
					\path [name intersections = {of = back arr and front rising pro , by = {rising interchanger square}}] ;
					\path [name intersections = {of = front falling pro and front rising pro , by = {interchanger disk}}] ;
					\node [bead] at (functor square) {μ A} ;
					\node [bead] at (falling interchanger square) {β \tuple*{A , m}} ;
					\node [bead] at (rising interchanger square) {α \tuple*{m , A}} ;
					\node [bead] (align) at (interchanger disk) {γ m} ;
				\end{tikzpicture}
				\, = \,
				\begin{tikzpicture}[string diagram , x = {(30mm , 0mm)} , y = {(0mm , -20mm)} , baseline=(align.base)]
					\coordinate (cell) at (1/4 , 1/2) ;
					\coordinate (left top) at (0 , 1/6) ;
					\coordinate (right bot) at (1 , 5/6) ;
					\draw [name path =  back arr]
						(cell |- 0 , 0) to [out = south , in = north] node [fromabove] {α \tuple*{B , A}} node [tobelow] {β \tuple*{A , B}}
						(cell |- 1 , 1) ;
					\draw [name path = back pro]
						(left top |- cell) to [out = east , in = west] node [fromleft] {γ B} node [toright] {δ A}
						(right bot |- cell) ;
					\draw [name path = front falling pro]
						(left top) to [out = east , in = west] node [fromleft] {F \tuple*{B , m}}
						(cell |- left top) to [out = east , in = west] node [toright] {J \tuple*{A , m}}
						(right bot) ;
					\draw [name path = front rising pro]
						(left top |- right bot) to [out = east , in = west] node [fromleft] {I \tuple*{m , B}}
						(cell |- right bot) to [out = east , in = west] node [toright] {G \tuple*{m , A}}
						(right bot |- left top) ;
					\path [name intersections = {of = back arr and back pro , by = {functor square}}] ;
					\path [name intersections = {of = back arr and front falling pro , by = {falling interchanger square}}] ;
					\path [name intersections = {of = back arr and front rising pro , by = {rising interchanger square}}] ;
					\path [name intersections = {of = front falling pro and front rising pro , by = {interchanger disk}}] ;
					\node [bead] at (functor square) {μ B} ;
					\node [bead] at (falling interchanger square) {α \tuple*{B , m}} ;
					\node [bead] at (rising interchanger square) {β \tuple*{m , B}} ;
					\node [bead] (align) at (interchanger disk) {δ m} ;
				\end{tikzpicture}
			$$
			assuming a consistent proarrow bracketing on the two sides of the equation.
	\end{description}
\end{definition}

We can diagonalize a modification to a dimodification
by \refer[transformation diagonalization]{diagonalizing} its proarrow-dimension natural transformations to dinatural transformations.

\begin{proposition}[modification diagonalization] \label{theorem: modification diagonalization}
	For each \refer{modification} of natural transformations of difunctors
	$
		μ : \doublehom[\doublehom[\hom[]{\product{\co{\cat{C}} , \cat{C}}}{\cat{D}}]
		{\prohom{F}{I}}{\prohom{G}{J}}{\hom{F}{G}}{\hom{I}{J}}]
		{γ}{δ}{α}{β}
	$
	there is a \refer{dimodification}
	$
		\di{μ} : \doublehom[\doublehom[\hom[]{\product{\co{\cat{C}} , \cat{C}}}{\cat{D}}]
		{\prohom<dinatural>{F}{I}}{\prohom<dinatural>{G}{J}}{\hom{F}{G}}{\hom{I}{J}}]
		{\di{γ}}{\di{δ}}{α}{β}
	$,
	which we call its \define[modification diagonalization]{diagonalization}, with
	\refer[dimodification object-component square]{object component squares}
	\[
		\begin{tikzpicture}[string diagram , x = {(28mm , 0mm)} , y = {(0mm , -12mm)} , baseline=(align.base)]
			\coordinate (cell) at (1/2 , 1/2) ;
			\draw [name path = back arr]
				(cell |- 0 , 0) to node [fromabove] {α \tuple*{A , A}} node [tobelow] {β \tuple*{A , A}}
				(cell |- 1 , 1) ;
			\draw [name path = back pro]
				(0 , 0 |- cell) to node [fromleft] {\di{γ} A} node [toright] {\di{δ} A}
				(1 , 1 |- cell) ;
			\path [name intersections = {of = back arr and back pro , by = {square}}] ;
			\node at ($ (square) + (-1/3 , -1/3) $) {F \tuple*{A , A}} ;
			\node at ($ (square) + (1/3 , -1/3) $) {G \tuple*{A , A}} ;
			\node at ($ (square) + (-1/3 , 1/3) $) {I \tuple*{A , A}} ;
			\node at ($ (square) + (1/3 , 1/3) $) {J \tuple*{A , A}} ;
			\node [bead] (align) at (square) {\di{μ} A} ;
		\end{tikzpicture}
		\; ≔ \;
		\begin{tikzpicture}[string diagram , x = {(26mm , 0mm)} , y = {(0mm , -12mm)} , baseline=(align.base)]
			\coordinate (cell) at (1/2 , 1/2) ;
			\draw [name path = back arr]
				(cell |- 0 , 0) to node [fromabove] {α \tuple*{A , A}} node [tobelow] {β \tuple*{A , A}}
				(cell |- 1 , 1) ;
			\draw [name path = back pro]
				(0 , 0 |- cell) to node [fromleft] {γ \tuple*{A , A}} node [toright] {δ \tuple*{A , A}}
				(1 , 1 |- cell) ;
			\path [name intersections = {of = back arr and back pro , by = {square}}] ;
			\node at ($ (square) + (-1/3 , -1/3) $) {F \tuple*{A , A}} ;
			\node at ($ (square) + (1/3 , -1/3) $) {G \tuple*{A , A}} ;
			\node at ($ (square) + (-1/3 , 1/3) $) {I \tuple*{A , A}} ;
			\node at ($ (square) + (1/3 , 1/3) $) {J \tuple*{A , A}} ;
			\node [bead] (align) at (square) {μ \tuple*{A , A}} ;
		\end{tikzpicture}
		\tag*{\hyperlink{proof: modification diagonalization}{\qedsymbol}}
	\]
\end{proposition}

\begin{corollary}[dummy dimodifications] \label{theorem: dummy dimodifications}
	For each \refer{modification}
	$
		μ : \doublehom[\doublehom[\hom[]{\cat{C}}{\cat{D}}]
		{\prohom{F}{I}}{\prohom{G}{J}}{\hom{F}{G}}{\hom{I}{J}}]
		{γ}{δ}{α}{β}
	$
	there is a \refer{dimodification} of dummy
	\refer[dummy natural transformation]{natural} and
	\refer[dummy dinatural transformation]{dinatural} transformations
	$
		\di{\dummy{μ}}  : \doublehom[\doublehom[\hom[]{\product{\co{\cat{C}} , \cat{C}}}{\cat{D}}]
		{\prohom<dinatural>{\dummy{F}}{\dummy{I}}}{\prohom<dinatural>{\dummy{G}}{\dummy{J}}}{\hom{\dummy{F}}{\dummy{G}}}{\hom{\dummy{I}}{\dummy{J}}}]
		{\di{\dummy{γ}}}{\di{\dummy{δ}}}{\dummy{α}}{\dummy{β}}
	$,
	which we call a \define{dummy dimodification}, with
	\refer[dimodification object-component square]{object-component squares} $\di{\dummy{μ}} A ≔ μ A$.
	\[
		\begin{tikzpicture}[string diagram , x = {(28mm , 0mm)} , y = {(0mm , -12mm)} , baseline=(align.base)]
			\coordinate (cell) at (1/2 , 1/2) ;
			\draw [name path = back arr]
				(cell |- 0 , 0) to node [fromabove] {\dummy{α} \tuple*{A , A}} node [tobelow] {\dummy{β} \tuple*{A , A}}
				(cell |- 1 , 1) ;
			\draw [name path = back pro]
				(0 , 0 |- cell) to node [fromleft] {\di{\dummy{γ}} A} node [toright] {\di{\dummy{δ}} A}
				(1 , 1 |- cell) ;
			\path [name intersections = {of = back arr and back pro , by = {square}}] ;
			\node at ($ (square) + (-1/3 , -1/3) $) {\dummy{F} \tuple*{A , A}} ;
			\node at ($ (square) + (1/3 , -1/3) $) {\dummy{G} \tuple*{A , A}} ;
			\node at ($ (square) + (-1/3 , 1/3) $) {\dummy{I} \tuple*{A , A}} ;
			\node at ($ (square) + (1/3 , 1/3) $) {\dummy{J} \tuple*{A , A}} ;
			\node [bead] (align) at (square) {\di{\dummy{μ}} A} ;
		\end{tikzpicture}
		\; ≔ \;
		\begin{tikzpicture}[string diagram , x = {(16mm , 0mm)} , y = {(0mm , -12mm)} , baseline=(align.base)]
			\coordinate (cell) at (1/2 , 1/2) ;
			\draw [name path = back arr]
				(cell |- 0 , 0) to node [fromabove] {α A} node [tobelow] {β A}
				(cell |- 1 , 1) ;
			\draw [name path = back pro]
				(0 , 0 |- cell) to node [fromleft] {γ A} node [toright] {δ A}
				(1 , 1 |- cell) ;
			\path [name intersections = {of = back arr and back pro , by = {square}}] ;
			\node at ($ (square) + (-1/3 , -1/3) $) {F A} ;
			\node at ($ (square) + (1/3 , -1/3) $) {G A} ;
			\node at ($ (square) + (-1/3 , 1/3) $) {I A} ;
			\node at ($ (square) + (1/3 , 1/3) $) {J A} ;
			\node [bead] (align) at (square) {μ A} ;
		\end{tikzpicture}
		\tag*{\hyperlink{proof: dummy dimodifications}{\qedsymbol}}
	\]
\end{corollary}

Continuing the pattern, a dummy dimodification ``is'' just a modification.

\begin{proposition} [dummy dimodifications are modifications] \label{theorem: dummy dimodifications are modifications}
	The \refer{dummy dimodification} construction determines a bijection
	between arbitrary \refer[modification]{modifications}
	and \refer[dimodification]{dimodifications}
	of dummy \refer[dummy natural transformation]{natural} and \refer[dummy dinatural transformation]{dinatural} transformations:
	\[
		\doublehom[\doublehom[\hom[]{\cat{C}}{\cat{D}}]
		{\prohom{F}{I}}{\prohom{G}{J}}{\hom{F}{G}}{\hom{I}{J}}]
		{γ}{δ}{α}{β}
		\quad ≅ \quad
		\doublehom[\doublehom[\hom[]{\product{\co{\cat{C}} , \cat{C}}}{\cat{D}}]
		{\prohom<dinatural>{\dummy{F}}{\dummy{I}}}{\prohom<dinatural>{\dummy{G}}{\dummy{J}}}{\hom{\dummy{F}}{\dummy{G}}}{\hom{\dummy{I}}{\dummy{J}}}]
		{\di{\dummy{γ}}}{\di{\dummy{δ}}}{\dummy{α}}{\dummy{β}}
		\tag*{\hyperlink{proof: dummy dimodifications are modifications}{\qedsymbol}}
	\]
\end{proposition}

Thus, by propositions
\ref{theorem: dummy difunctors are functors},
\ref{theorem: dummy dinatural transformations are natural transformations}, and
\ref{theorem: dummy dimodifications are modifications}
we may regard ordinary functors, natural transformations, and modifications of double categories
as special cases of difunctors, dinatural transformations, and dimodifications,
constituting those elements that are dummy in one factor of their dual pairing.

\paragraph{Related constructions}
Dubuc and Street introduced the concept of dinaturality for $1$-categories \cite{street-1970-dinatural_transformations}
and proved decategorified versions of the propositions about dinatural transformations for double categories above,
except for proposition \ref{theorem: groupoidal dinatural composition},
whose $1$-categorical analogue appears in \cite{laretto-2024-directed_type_theory}.

The term ``dinaturality'' is an abbreviation for ``\underline{di}agonal naturality'',
presumably because the components of a dinatural transformation use the same element for both arguments
to the boundary functors.
It is by back-formation from ``dinatural transformation'' that we use ``difunctor'' for a functor whose domain is a dual pairing
and ``dimodification'' for a single-object-indexed natural square with a pair of dinatural transformation boundaries.
Admittedly, difunctors are not ``diagonal'' in the same sense that dinatural transformations and dimodifications are,
but the parallel terminology is mnemonic, and if required can be justified as ``\underline{d}ual \underline{i}nput functors''.

Vidal and Tur extended the definition of dinatural transformation to $2$-categories,
including the pseudo, directed, and strict variants \cite{vidal-2010-2_categorical_institution}.
Their definition 2.25 seems to be missing morphism composition compatibility laws
corresponding to equations \eqref{proarrow-dimension dinatural transformation proarrow composition compatibility} above.
However, analogous laws are included later in their consideration of dinatural transformations of pseudo functors of $2$-categories.

%% file: content/parts/extranaturality.tex

The original motivation for investigating dinaturality for double categories
was in fact to study the phenomenon of extranaturality \cite{eilenberg-1966-functorial_calculus}.
Dubuc and Street decomposed extranatural transformations of functors of $1$-categories
into certain simple dinatural transformations \cite{street-1970-dinatural_transformations}.
Street later extended this construction to the general setting of a monoidal bicategory with duals \cite{street-2003-functorial_calculus}.
Here we characterize extranaturality for $1$-categories using dinaturality for double categories.

Throughout this section let $\cat{C}$ be a \refer{monoidal double category}
equipped with a functor $\dual{} : \hom{\co{\cat{C}}}{\cat{C}}$.
We define \refer[difunctor]{difunctors}
$I , \leftclass{\dualizer}{\tensor{,}} , \rightclass{\dualizer}{\tensor{,}} : \hom{\product{\co{\cat{C}} , \cat{C}}}{\cat{C}}$
as follows.
$$
	\tensor{} ≔ κ \tensor{}
	\quad , \quad
	\leftclass{\dualizer}{\tensor{,}} ≔ \comp{(\product{\dual{} , \id{}}) , \tensor{,}}
	\quad , \quad
	\rightclass{\dualizer}{\tensor{,}} ≔ \comp{(\product{\dual{} , \id{}}) , \braiding , \tensor{,}}
$$
That is, $\tensor{}$ is the constant functor to the monoidal unit,
and $\leftclass{\dualizer}{\tensor{,}}$ and $\rightclass{\dualizer}{\tensor{,}}$ send the ordered pair
of a $\co{\cat{C}}$-square $ψ$
and a $\cat{C}$-square $φ$
respectively to the $\cat{C}$-squares
$\tensor{\dual{ψ} , φ}$ and $\tensor{φ , \dual{ψ}}$.
$$

$$

\begin{definition}[fold dinatural transformations] \label{definition: fold dinatural transformations}
	We call \refer[proarrow-dimension dinatural transformation]{proarrow-dimension dinatural transformations}
	$$
		η : \prohom<dinatural>[(\hom{\product{\co{\cat{C}} , \cat{C}}}{\cat{C}})]{\tensor{}}{\leftclass{\dualizer}{\tensor{,}}}
		\quad \text{and} \quad
		ε : \prohom<dinatural>[(\hom{\product{\co{\cat{C}} , \cat{C}}}{\cat{C}})]{\rightclass{\dualizer}{\tensor{,}}}{\tensor{}}
	$$
	respectively a \define{cap} and \define{cup} \define{fold} for $\tuple*{\cat{C} , \dual{}}$.
\end{definition}

Let us expand these definitions.
For $A : \cat{C}$
folds have the following \refer[proarrow-dimension dinatural transformation object-component proarrow]{object-component proarrows},
with their string diagrams in the $\tuple*{\tensor{,} ,  \procomp{,}}$-plane shown below.
$$
	η A : \prohom{\tensor{}}{\tensor{\dual{A} , A}}
	\qquad \text{and} \qquad
	ε A : \prohom{\tensor{A , \dual{A}}}{\tensor{}}
$$
\vspace{-2ex}
$$

$$
For $f : \hom[\cat{C}]{A}{B}$
folds have the following \refer[proarrow-dimension dinatural transformation arrow-component square]{arrow-component squares},
with their deprojected surface diagrams shown below.
$$
	%
$$
For $m : \prohom[\cat{C}]{A}{B}$
folds have the following \refer[invertible disk]{invertible} \refer[proarrow-dimension dinatural transformation proarrow-component disk]{proarrow-component disks},
with their deprojected surface diagrams shown below.
$$
	%
$$

Preservation of binary arrow composition \eqref{proarrow-dimension dinatural transformation arrow composition preservation}
says that the square formed by two consecutive arrows crossing a fold has a unique interpretation,
while compatibility with binary proarrow composition \eqref{proarrow-dimension dinatural transformation proarrow composition compatibility}
says that the square formed by two consecutive proarrows crossing a fold has a unique interpretation.
$$
	\centerdisplay{
	%
	}
$$

Preservation of nullary arrow composition \eqref{proarrow-dimension dinatural transformation arrow composition preservation} and
compatibility with nullary proarrow composition \eqref{proarrow-dimension dinatural transformation proarrow composition compatibility}
together imply that a ``blank'' cap or cup square has a unique interpretation.
$$
	%
$$

Naturality for squares \eqref{proarrow-dimension dinatural transformation square naturality}
says that a square with one of its arrow and proarrow boundaries crossing a fold can be ``flipped over'' to the other side.
$$
	\centerdisplay{
	%
	}
$$

Thus folds are dinatural transformations that ``bend surfaces around'',
but preserve structures embedded within those surfaces, at least to the extent that the functor $\dual{}$ does so.
\refer[transformation diagonalization]{Diagonalizing} the \refer[identity natural transformation]{identity natural transformations}
on the \refer[dummy difunctor]{dummy difunctors}
$π_1 \, , \, \comp{π_0 , \dual{}} : \hom{\product{\co{\cat{C}} , \cat{C}}}{\cat{C}}$
yields dinatural transformations that preserve surface geometry as well.

\begin{definition}[strip (di)natural transformations]
	We call the following proarrow-dimension \refer[dummy dinatural transformation]{dummy dinatural transformations}
	\define[strip]{strips}.
	$$
		 ι ≔ \di{\procomp{}} \, π_1 :
		 \prohom<dinatural>[(\hom{\product{\co{\cat{C}} , \cat{C}}}{\cat{C}})]{π_1}{π_1}
		\quad \text{and} \quad
		\dual{ι} ≔ \di{\procomp{}} \, (\comp{π_0 , \dual{}}) :
		\prohom<dinatural>[(\hom{\product{\co{\cat{C}} , \cat{C}}}{\cat{C}})]{\comp{π_0 , \dual{}}}{\comp{π_0 , \dual{}}}
	$$
\end{definition}

Note that while  $ι$ is dummy in its first argument, as in corollary \ref{theorem: dummy dinatural transformations},
$\dual{ι}$ is dummy in its second argument.
By proposition \ref{theorem: dummy dinatural transformations are natural transformations} (and its analogue dummy in the other factor)
these dinatural transformations have the same structure and laws as the identity natural transformations
$\procomp{} \, \comp{} : \prohom[\hom[]{\cat{C}}{\cat{C}}]{\comp{}}{\comp{}}$ and
$\procomp{} \, \dual{} : \prohom[\hom[]{\co{\cat{C}}}{\cat{C}}]{\dual{}}{\dual{}}$.
They have
\refer[proarrow-dimension dinatural transformation object-component proarrow]{object-component proarrows}
$\procomp{} \, A$ and $\procomp{} \, \dual{A}$,
\refer[proarrow-dimension dinatural transformation arrow-component square]{arrow-component squares}
$\procomp{} \, f$ and $\procomp{} \, \dual{f}$, and
\refer[proarrow-dimension dinatural transformation proarrow-component disk]{proarrow-component disks} (isomorphic to)
$\comp{} \, m$ and $\comp{} \, \dual{m}$, respectively.
Preservation of arrow composition \eqref{proarrow-dimension dinatural transformation arrow composition preservation}
and compatibility with proarrow composition \eqref{proarrow-dimension dinatural transformation proarrow composition compatibility}
say that there is a unique interpretation of a strip containing multiple arrows or proarrows,
while naturality for squares \eqref{proarrow-dimension dinatural transformation square naturality}
says that squares can slide past strips.
$$

$$

Folds and strips constitute the basic building blocks of extranaturality.
We call these dinatural transformations \define[connector]{basic connectors}.
By taking tensor products of basic connectors
we can form compound connectors with object-component proarrows such as
$$
	(\tensor{ε , ι , η}) \, \tuple*{A , B , C} : \prohom{\tensor{A , \dual{A} , B}}{\tensor{B , \dual{C} , C}}
$$
$$
	%
$$

Unlike dinatural transformations in general, compatible configurations of these connectors can be composed.
Moreover, by proposition \ref{theorem: dummy dinatural transformations are natural transformations},
in cases where the composite boundaries are dummy difunctors they compose to ordinary natural transformations.

\begin{proposition}[zigzag natural transformations]  \label{theorem: zigzag natural transformations}
	For \refer{cap} $η$ and \refer{cup} $ε$
	there are \refer[proarrow-dimension natural transformation]{proarrow-dimension natural transformations}
	$$
		S : \prohom[\hom[]{\cat{C}}{\cat{C}}]{\comp{}}{\comp{}}
		\quad \text{and} \quad
		Z : \prohom[\hom[]{\co{\cat{C}}}{\cat{C}}]{\dual{}}{\dual{}}
	$$
	that we call right and left \define[zigzag]{zigzags}
	with the following structure.
	\begin{description}
		\item[{\refer[proarrow-dimension natural transformation object-component proarrow]{object-component proarrows}}:]
			for $A : \cat{C}$, we define $S A$ and $Z A$ as
			$$
				\procomp{(\tensor{ι A , η A}) , (\tensor{ε A , ι A})} : \prohom[\cat{C}]{A}{A}
				\quad \text{and} \quad
				\procomp{(\tensor{η A , \dual{ι} A}) , (\tensor{\dual{ι} A , ε A})} : \prohom[\cat{C}]{\dual{A}}{\dual{A}}
			$$
			$$

			$$
		\item[{\refer[proarrow-dimension natural transformation arrow-component square]{arrow-component squares}}:]
			for $f : \hom[\cat{C}]{A}{B}$ we define $S f$ and $Z f$ as
			$$
				%
			$$
		\item[{\refer[proarrow-dimension natural transformation proarrow-component disk]{proarrow-component disks}}:]
			for $m : \prohom[\cat{C}]{A}{B}$ we define $S m$ and $Z m$ as
			$$
				\centerdisplay{
				%
				\tag*{\hyperlink{proof: zigzag natural transformations}{\qedsymbol}}
			\]
	\end{description}
\end{proposition}

Note that an \refer{oplax dinatural transformation} \refer[dummy dinatural transformation]{dummy} in its second argument
corresponds to a \refer{lax natural transformation}.
This explains the orientation of $Z$'s proarrow-component disks.
But because all of our transformations are pseudo, the only difference this makes here
is which direction of an invertible proarrow-component disk is ``forward''.

Zigzags are each formed from two compatible folds.
They have the same proarrow-dimension boundary as the \refer[identity natural transformation]{identity natural transformations}
$\procomp{} \, \comp{}$ and $\procomp{} \, \dual{}$, respectively,
suggesting the utility of a structure to mediate between them.

\begin{definition}[cusp modifications]
	For a \refer{cap} $η$ and \refer{cup} $ε$
	we call \refer[invertible modification]{invertible} \refer[globular modification]{globular} \refer[modification]{modifications}
	$$
		ζ : \hom[\prohom[\hom[]{\cat{C}}{\cat{C}}]{\comp{}}{\comp{}}]{S}{\procomp{} \, \comp{}}
		\quad \text{and} \quad
		ξ : \hom[\prohom[\hom[]{\co{\cat{C}}}{\cat{C}}]{\dual{}}{\dual{}}]{Z}{\procomp{} \, \dual{}}
	$$
	respectively a right and left zigzag \define{cusp}.
\end{definition}

Let us expand these definitions.
For $A : \cat{C}$
cusps have the following \refer[invertible disk]{invertible} \refer[modification object-component square]{object-component proarrow disks}.
$$
	ζ A : \hom{S A}{\procomp{} \, A}
	\qquad \text{and} \qquad
	ξ A : \hom{Z A}{\procomp{} \, \dual{A}}
$$
\vspace{-2ex}
$$

$$
Naturality for arrows \eqref{modification arrow naturality}
says that an arrow may ``slip off'' a zigzag across a cusp.
$$
	\centerdisplay{
	%
	}
$$
Naturality for proarrows \eqref{modification proarrow naturality}
says the same about them.
$$
	\centerdisplay{
	%
	}
$$

\begin{remark}[proarrows in zigzags] \label{theorem: proarrows in zigzags}
	By precomposing the last equations with $\inv{(ε m)}$ and $\inv{(η m)}$ respectively,
	we obtain coherent isomorphisms between a proarrow and its $\dual{}$-image,
	with one of the two ``in a zigzag''.
	$$
		\procomp{(\tensor{A , η B}) , (\tensor{A , \dual{m} , B}) , (\tensor{ε A , B})} \; ≅ \; m
		\quad , \quad
		\procomp{(\tensor{ηA , \dual{B}}) , (\tensor{\dual{A} , m , \dual{B}}) , (\tensor{\dual{A} , ε B})} \; ≅ \; \dual{m}
	$$
	$$
	\centerdisplay{
		%
		}
	$$
\end{remark}

Cusps straighten zigzags by eliminating two consecutive folds.
They reify the laws of an adjunction as (in this case, invertible) structure.
The chirality in the naming of zigzags and their cusps corresponds to that of the adjoints on the boundary.

As usual, when reifying a property as a structure we wish to do so \emph{coherently}
by introducing new properties governing the new structure.
Several versions of a categorified adjunction can be defined,
differing in their ambient setting, variance, and strictness.
These include
Verity's ``locally adjoint biadjoint pairs'' in Gray categories \cite{verity-1992-thesis},
Gurski's ``biadjunctions'' in tricategories \cite{gurski-2011-biequivalences_in_tricategories}, and
Pstrągowski's ``coherent dual pairs'' in monoidal bicategories \cite{pstragowski-2022-dualizable_objects_in_monoidal_bicategories}.
We will refer to the variant encountered here as ``adjoint folds''.

In order to state the coherence relations for cusps
we need configurations of \refer[connector]{connectors}
that can be simplified by cusps in multiple ways.

\begin{proposition}[sesqui-zigzag dinatural transformations]  \label{theorem: sesqui-zigzag dinatural transformations}
	For \refer{cap} $η$ and \refer{cup} $ε$
	there are \refer[proarrow-dimension dinatural transformation]{proarrow-dimension dinatural transformations}
	$$
		M : \prohom<dinatural>[(\hom{\product{\co{\cat{C}} , \cat{C}}}{\cat{C}})]{\tensor{}}{\leftclass{\dualizer}{\tensor{,}}}
		\quad \text{and} \quad
		W : \prohom<dinatural>[(\hom{\product{\co{\cat{C}} , \cat{C}}}{\cat{C}})]{\rightclass{\dualizer}{\tensor{,}}}{\tensor{}}
	$$
	with {\refer[proarrow-dimension dinatural transformation object-component proarrow]{object-component proarrows}}
	$$
		\begin{tikzpicture}[string diagram , x = {(18mm , 0mm)} , y = {(0mm , -8mm)} , baseline=(align.base)]
			\coordinate (l) at (0 , 1/2) ;
			\coordinate (cl) at (1/3 , 1/2) ;
			\coordinate (cr) at (2/3 , 1/2) ;
			\coordinate (r) at (1 , 1/2) ;
			\draw
				(l |- 1 , 1) to [out = north , in = south] node [frombelow] {\dual{A}}
				(l) to [out = north , looseness = 2 , in = north] coordinate [pos = 1/2] (unit 1)
				(cl) to [out = south , looseness = 2 , in = south] coordinate [pos = 1/2] (counit)
				(cr) to [out = north , looseness = 2 , in = north] coordinate [pos = 1/2] (unit 2)
				(r) to [out = south , in = north] node [tobelow] {A}
				(r |- 1 , 1) ;
			\node [bead] at (unit 1) {η A} ;
			\node [bead] at (counit) {ε A} ;
			\node [bead] at (unit 2) {η A} ;
			\node (align) at (cl) {} ;
		\end{tikzpicture}
		\qquad \text{and} \qquad
		\begin{tikzpicture}[string diagram , x = {(18mm , 0mm)} , y = {(0mm , -8mm)} , baseline=(align.base)]
			\coordinate (l) at (0 , 1/2) ;
			\coordinate (cl) at (1/3 , 1/2) ;
			\coordinate (cr) at (2/3 , 1/2) ;
			\coordinate (r) at (1 , 1/2) ;
			\draw
				(l |- 0 , 0) to [out = south , in = north] node [fromabove] {A}
				(l) to [out = south , looseness = 2 , in = south] coordinate [pos = 1/2] (counit 1)
				(cl) to [out = north , looseness = 2 , in = north] coordinate [pos = 1/2] (unit)
				(cr) to [out = south , looseness = 2 , in = south] coordinate [pos = 1/2] (counit 2)
				(r) to [out = north , in = south] node [toabove] {\dual{A}}
				(r |- 0 , 0) ;
			\node [bead] at (counit 1) {ε A} ;
			\node [bead] at (unit) {η A} ;
			\node [bead] at (counit 2) {ε A} ;
			\node (align) at (cl) {} ;
		\end{tikzpicture}
	$$
	and the obvious
	{\refer[proarrow-dimension dinatural transformation arrow-component square]{arrow-component squares}} and
	{\refer[proarrow-dimension dinatural transformation proarrow-component disk]{proarrow-component disks}}.
	\hfill \hyperlink{proof: sesqui-zigzag dinatural transformations}{\qedsymbol}
\end{proposition}

Sesqui-zigzags are themselves caps and cups according to definition \ref{definition: fold dinatural transformations},
but they are decomposable into the the specified cap $η$ and cup $ε$.
Moreover, up to \refer{monoidal naturality}, each contains both a right and a left zigzag,
which can be straightened by the corresponding cusp.

\begin{definition}[adjoint folds]
	Cap and cup \refer[fold]{folds} $η$ and $ε$,
	together with right and left zigzag \refer[cusp]{cusps} $ζ$ and $ξ$
	form \define{adjoint folds}
	if the following sesqui-zigzag reducing coherence laws hold.
	\begin{equation} \label{swallowtail coherence}
		\procomp{η , (\tensor{\dual{ι} , ζ})} ≅ \procomp{η , (\tensor{ξ , ι})}
		\qquad , \qquad
		\procomp{(\tensor{ι , ξ}) , ε} ≅ \procomp{(\tensor{ζ , \dual{ι}}) , ε}
	\end{equation}
\end{definition}

The coherence laws, equivalent to the ``swallowtail relations'' in this setting,
say that the two ways of simplifying an $M$ to an $η$, and a $W$ to an $ε$, using either $ζ$ or $ξ$, are equivalent.
These are perhaps easiest to understand as equations between surface diagrams
expressed in the form of commuting diagrams between rewritings of slice string diagrams.
$$

$$
These laws imply that
proarrows comprising any whole- or half-integral number of zigzag components,
which we call \define[folding]{foldings},
are determined up to coherent simplification by cusp components (and proarrow-composition coherators)
by their boundaries.
They are coherently isomorphic to components of either a \refer{strip} or a \refer{fold}
as depicted below, where an ellipsis represents an integral number of zigzag components.
$$
	\centerdisplay{
	%
	}
$$

We refer to these coherently mediating composites of whiskered cusps as \define[folding coherator]{folding coherators}.
For the sake of legibility, we typically elide folding coherators from the boundaries of diagrams,
and refer to such diagrams as \define{ragged}.
As in the case of proarrow composition coherators, equations that hold up to implicit folding coherators are indicated by ``$≅$''.

As sketched by Kelly and Street \cite{kelly-1974-2_categories},
and elaborated by Cheng, Gurski, and Riehl \cite{riehl-2014-multivariable_adjunctions},
the ``conjugate'' or ``mate'' correspondence of $2$-cells under a pair of adjunctions \cite{mac_lane-1998-categories}
has a convenient double-categorical presentation.
Here, we consider this correspondence in the setting of a monoidal double category with adjoint folds.

\begin{definition}[mates]
	For arrows $f : \hom{A}{B}$ and $g : \hom{C}{D}$ and proarrows $m : \prohom{A}{C}$ and $n : \prohom{B}{D}$
	in a monoidal double category with adjoint folds,
	the left and right \define{mate} functions
	$$
		\leftmate{\arg} : \hom{\doublehom{m}{n}{f}{g}}{\doublehom{\dual{m}}{\dual{n}}{\dual{g}}{\dual{f}}}
		\qquad \text{and} \qquad
		\rightmate{\arg} : \hom{\doublehom{\dual{m}}{\dual{n}}{\dual{g}}{\dual{f}}}{\doublehom{m}{n}{f}{g}}
	$$
	are defined as follows.
	$$

	$$
\end{definition}

The chirality in the naming of the mates corresponds to that of the adjoints on the boundary.
Mates have the following \refer{ragged} surface diagrams.
$$
	%
$$

\begin{lemma}[mate bijection] \label{theorem: mate bijection}
	For arrows $f$ and $g$ and proarrows $m$ and $n$,
	the mate functions form a bijection of squares
	$
		\doublehom{m}{n}{f}{g}
		≅
		\doublehom{\dual{m}}{\dual{n}}{\dual{g}}{\dual{f}}
	$.
	\hfill \hyperlink{proof: mate bijection}{\qedsymbol}
\end{lemma}

Because functors are single-valued,
a square $φ$ has exactly one left mate, whose boundary is the $\dual{}$-image of that of $φ$.
However, it has one right mate for each $\dual{}$-preimage of $φ$'s boundary.
It follows that the left mate of a square (at the only possible boundary) is its $\dual{}$-image;
and that the right mate of its $\dual{}$-image at the original boundary, as well as the $\dual{}$-image of any right mate, are both the square itself.

\begin{lemma}[mates and $\dual{}$] \label{theorem: mates and *}
	\liststartstheorem
	\begin{itemize}
		\item
			For a square $φ : \doublehom{m}{n}{f}{g}$ we have
			$\dual{φ} = \leftmate{φ} : \doublehom{\dual{m}}{\dual{n}}{\dual{g}}{\dual{f}}$ and
			$φ = \rightmate{(\dual{φ})} : \doublehom{m}{n}{f}{g}$.
		\item
			For a square $ψ : \doublehom{\dual{m}}{\dual{n}}{\dual{g}}{\dual{f}}$ we have
			$ψ = \dual{(\rightmate{ψ})}$ \; (taking $\rightmate{ψ} : \doublehom{m}{n}{f}{g}$).
			\hfill \hyperlink{proof: mates and *}{\qedsymbol}
	\end{itemize}
\end{lemma}

The mate functions on double-hom sets are compatible with the composition structure of double categories.

\begin{lemma}[mates of identity disks] \label{theorem: mates of identity disks}
	The mate functions on double-hom sets
	are compatible with nullary square composition in both dimensions,
	in the sense that for arrow $f$ and proarrow $m$
	the following equations hold when the boundaries agree.
	\[
		\leftmate{(\comp{} \, m)} = \comp{} \, \dual{m}
		\quad , \quad
		\leftmate{(\procomp{} \, f)} = \procomp{} \, \dual{f}
		\quad , \quad
		\rightmate{(\comp{} \, \dual{m})} = \comp{} \, m
		\quad , \quad
		\rightmate{(\procomp{} \, \dual{f})} = \procomp{} \, f
		\tag*{\hyperlink{proof: mates of identity disks}{\qedsymbol}}
	\]
\end{lemma}

\begin{lemma}[mates of proarrow-dimension composite squares] \label{theorem: mates of proarrow-dimension composite squares}
	The mate functions on double-hom sets
	are compatible with binary proarrow-dimension square composition,
	in the sense that for consecutive squares $φ$ and $ψ$
	the following equations hold when the boundaries agree.
	\[
		\leftmate{(\procomp{φ , ψ})} = \procomp{\leftmate{ψ} , \leftmate{φ}}
		\quad , \quad
		\rightmate{(\procomp{φ , ψ})} = \procomp{\rightmate{ψ} , \leftmate{φ}}
		\tag*{\hyperlink{proof: mates of proarrow-dimension composite squares}{\qedsymbol}}
	\]
\end{lemma}

\begin{lemma}[mates of arrow-dimension composite squares] \label{theorem: mates of arrow-dimension composite squares}
	The mate functions on double-hom sets
	are compatible with binary arrow-dimension square composition,
	in the sense that for consecutive squares $φ$ and $ψ$
	the following equations hold when the boundaries agree.
	\[
		\leftmate{(\comp{φ , ψ})} = \comp{\leftmate{φ} , \leftmate{ψ}}
		\quad , \quad
		\rightmate{(\comp{φ , ψ})} = \comp{\rightmate{φ} , \rightmate{ψ}}
		\tag*{\hyperlink{proof: mates of arrow-dimension composite squares}{\qedsymbol}}
	\]
\end{lemma}

By lemmas \ref{theorem: mates and *} through \ref{theorem: mates of arrow-dimension composite squares}
we conclude,

\begin{proposition} \label{theorem: * as left mate}
	For a monoidal double category with adjoint folds,
	the left mate function on squares determines a functor of double categories
	that coincides with $\co{\dual{}} : \hom{\cat{C}}{\co{\cat{C}}}$.
\end{proposition}

A functor of double categories $F : \hom{\cat{C}}{\cat{D}}$ is \define{fully faithful on squares} if
$\bind[∀]{ψ : \doublehom[\cat{D}]{F m}{F n}{F f}{F g}}{\bind[∃ !]{φ : \doublehom[\cat{C}]{m}{n}{f}{g}}{F φ = ψ}}$.

\begin{proposition} \label{theorem: * functor fully faithful on squares}
	For a monoidal double category with adjoint folds,
	the functor $\dual{} : \hom{\co{\cat{C}}}{\cat{C}}$ is fully faithful on squares.
	\hfill \hyperlink{proof: * functor fully faithful on squares}{\qedsymbol}
\end{proposition}

Adjoint $1$-cells in a $2$-category determine one another uniquely up to canonical isomorphism \cite{kelly-1974-2_categories}.
This result extends to the $0$- and $1$-cells of adjoint folds in a monoidal double category.

\begin{lemma}[] \label{theorem: same * 1}
	For a monoidal double category with adjoint folds we have the following.
	\begin{itemize}
		\item
			For objects $A_0$ and $A_1$
			with $\dual{A_0} = \dual{A_1}$,
			let $S \tuple*{A_0 , A_1}$ be the following proarrow.
			$$

			$$
			There are nested \refer[folding coherator]{folding coherators}
			$$
				\procomp{S \tuple*{A_1 , A_0} , S \tuple*{A_0 , A_1}} ≅ \procomp{} \, A_1
			$$
		\item
			For arrows $f_0$ and $f_1$
			with $\dual{f_0} = \dual{f_1}$,
			let $S \tuple*{f_0 , f_1}$ be the following square.
			$$
				%
			$$
			Conjugating by the nested folding coherators above we have
			$$
				\procomp{S \tuple*{f_1 , f_0} , S \tuple*{f_0 , f_1}} ≅ \procomp{} \, f_1
			$$
		\item
			For proarrows $m_0$ and $m_1$
			with $\dual{m_0} = \dual{m_1}$,
			let $S \tuple*{m_0 , m_1}$ be the following disk.
			$$
				%
			$$
			Conjugating by the nested folding coherators above we have
			\[
				\comp
				{
					(\procomp{S \tuple*{m_1 , m_0} , \comp{} \, S \tuple*{B_0 , B_1}})
					,
					(\procomp{\comp{} \, S \tuple*{A_1 , A_0} , S \tuple*{m_0 , m_1}})
				}
				\; ≅ \;
				\comp{} \, m_1
				\tag*{\hyperlink{proof: same * 1}{\qedsymbol}}
			\]
	\end{itemize}
\end{lemma}

\begin{proposition} \label{theorem: same * 2}
	For squares $φ_0 : \doublehom{m_0}{n_0}{f_0}{g_0}$ and $φ_1 : \doublehom{m_1}{n_1}{f_1}{g_1}$
	in a monoidal double category with adjoint folds
	if $\dual{φ_0} = \dual{φ_1}$
	then
	\[
		%
		\tag*{\hyperlink{proof: same * 2}{\qedsymbol}}
	\]
\end{proposition}

Letting $\dual{φ_0} ≔ ψ ≕ \dual{φ_1}$,
we have $\rightmate{ψ} : \doublehom{m_1}{n_1}{f_1}{g_1}$
expressed in terms of $\rightmate{ψ} : \doublehom{m_0}{n_0}{f_0}{g_0}$.
That is, any two right mates of a square are interdefinable
by conjugation with boundary mediating squares $S \tuple*{\arg , \arg}$ and folding coherators.

The motivation for defining the structure of \refer{adjoint folds} on a monoidal double category
is that $\Cat{Cat}$ has such structure,
and we can use it to define extranatural transformations there.
Intuitively, \refer[fold]{folds} in $\Cat{Cat}$ are defined to be (isomorphic to) the \refer[strip]{strips} obtained by ``unfolding'' them.
Graphically, this corresponds to pivoting their $\dual{}$-ed surfaces through a half-turn along caps and cups
and removing the ``$\dual{}$''s from the cells in the pivoted surfaces.
This transforms the cusps into identities.

\begin{proposition}[adjoint folds in $\Cat{Cat}$] \label{theorem: adjoint folds in Cat}
	The cartesian monoidal weak double category of (suitably small) $1$-categories,
	functors, profunctors, and natural transformations, $\Cat{Cat}$,
	equipped with the oppositization functor, $\op{} : \hom{\co{\Cat{Cat}}}{\Cat{Cat}}$,
	has \refer{cap} and \refer{cup}
	$$
		η : \prohom<dinatural>[(\hom{\product{\co{\Cat{Cat}} , \Cat{Cat}}}{\Cat{Cat}})]{\cat{\product{}}}{\leftclass{\oppositizer}{\product{,}}}
		\quad , \quad
		ε : \prohom<dinatural>[(\hom{\product{\co{\Cat{Cat}} , \Cat{Cat}}}{\Cat{Cat}})]{\rightclass{\oppositizer}{\product{,}}}{\cat{\product{}}}
	$$
	that form \refer{adjoint folds}
	with the following structure.
	\begin{description}
		\item[{\refer[proarrow-dimension dinatural transformation object-component proarrow]{object-component proarrows}}:]
			for each category $\cat{A}$,
			profunctors
			$$
				η \cat{A} : \prohom{\cat{\product{}}}{\product{\op{\cat{A}} , \cat{A}}}
				\; :≅ \;
				\procomp{} \, \cat{A} : \prohom{\cat{A}}{\cat{A}}
				\; ≅: \;
				ε \cat{A} : \prohom{\product{\cat{A} , \op{\cat{A}}}}{\cat{\product{}}}
			$$
			i.e., the cap and cup are both the hom functor
			$\hom[\cat{A}]{\arg}{\arg} : \hom{\product{\op{\cat{A}} , \cat{A}}}{\Cat{Set}}$,
		\item[{\refer[proarrow-dimension dinatural transformation arrow-component square]{arrow-component squares}}:]
			for each functor $F : \hom{\cat{A}}{\cat{B}}$,
			natural transformations
			$$
				η F : \doublehom{η \cat{A}}{η \cat{B}}{\comp{} \, \cat{\product{}}}{\product{\op{F} , F}}
				\; :≅ \;
				\procomp{} \, F : \prohom{F}{F}
				\; ≅: \;
				ε F : \doublehom{ε \cat{A}}{ε \cat{B}}{\product{F , \op{F}}}{\comp{} \, \cat{\product{}}}
			$$
			i.e., the cap and cup are both the natural transformation from $\cat{A}$-homs to $\cat{B}$-homs of $F$-images of objects
			that applies $F$ to the arrows,
			$\comp{\arg , F} : \hom{\hom[\cat{A}]{\arg}{\arg}}{\hom[\cat{B}]{F \arg}{F \arg}}$,
		\item[{\refer[proarrow-dimension dinatural transformation proarrow-component disk]{proarrow-component disks}}:]
			for each profunctor $M : \prohom{\cat{A}}{\cat{B}}$,
			natural isomorphisms
			$$
				η M : \hom{\procomp{η \cat{B} , (\product{\op{M} , \procomp{} \, \cat{B}})}}{\procomp{η \cat{A} , (\product{\procomp{} \, \op{\cat{A}} , M})}}
				\; :≅ \;
				\comp{} \, M : \hom{M}{M}
			$$
			$$
				≅: \;
				ε M : \hom{\procomp{(\product{M , \procomp{} \, \op{\cat{B}}}) , ε \cat{B}}}{\procomp{(\product{\procomp{} \, \cat{A} , \op{M}}) , ε \cat{A}}}
			$$
			i.e., the cap and cup are both the identity natural transformation on the given profunctor;
			recalling that $\op{M} \tuple*{B , A} = M \tuple*{A , B}$, where $\op{M} : \prohom{\op{\cat{B}}}{\op{\cat{A}}}$ and
			\[
				\prohom{\op{\cat{B}}}{\op{\cat{A}}}
				\, ≔ \,
				\hom{\product{\cat{B} , \op{\cat{A}}}}{\Cat{Set}}
				\, ≅ \,
				\hom{\product{\op{\cat{A}} , \cat{B}}}{\Cat{Set}}
				\, ≕ \,
				\prohom{\cat{A}}{\cat{B}}
				\tag*{\hyperlink{proof: adjoint folds in Cat}{\qedsymbol}}
			\]
	\end{description}
\end{proposition}

The classical definition of an extranatural transformation (or ``generalized natural transformation'')
is given in terms of a family of component morphisms
constrained by a trio of (generalized) naturality relations.

\begin{definition}[extranatural transformation \cite{eilenberg-1966-functorial_calculus}]
	Given functors of $1$-categories
	$F : \hom{\product{\cat{A} , \op{\cat{A}} , \cat{B}}}{\cat{D}}$ and
	$G : \hom{\product{\cat{B} , \op{\cat{C}} , \cat{C}}}{\cat{D}}$
	an \define{extranatural transformation} $α : \hom<extranatural>{F}{G}$
	consists of the following data.
	\begin{description}
		\item[{\define*[extranatural transformation object-component arrow]{object-component arrows}}:]
			for each $(\product{\cat{A} , \cat{B} , \cat{C}})$-object
			$\tuple*{A , B , C}$
			a $\cat{D}$-arrow
			$α \tuple*{A , B , C} : \hom{F \tuple*{A , A , B}}{G \tuple*{B , C , C}}$.
	\end{description}
	This data is required to satisfy the following relations.
	\begin{description}
		\item[domain extranaturality in $\cat{A}$:]
			for an arrow $f : \hom[\cat{A}]{A}{A′}$
			we have
			\begin{equation} \label{extranatural transformation domain extranaturality}
				\comp{F \tuple*{f , A′ , B} , α \tuple*{A′ , B , C}}
				\; = \;
				\comp{F \tuple*{A , f , B} , α \tuple*{A , B , C}}
			\end{equation}
		\item[ordinary naturality in $\cat{B}$:]
			for an arrow $g : \hom[\cat{B}]{B}{B′}$
			we have
			\begin{equation} \label{extranatural transformation ordinary naturality}
				\comp{F \tuple*{A , A , g} , α \tuple*{A , B′ , C}}
				\; = \;
				\comp{α \tuple*{A , B , C} , G \tuple*{g , C , C}}
			\end{equation}
		\item[codomain extranaturality in $\cat{C}$:]
			for an arrow $h : \hom[\cat{C}]{C}{C′}$
			we have
			\begin{equation} \label{extranatural transformation codomain extranaturality}
				\comp{α \tuple*{A , B , C′} , G \tuple*{B , h , C′}}
				\; = \;
				\comp{α \tuple*{A , B , C} , G \tuple*{B , C , h}}
			\end{equation}
	\end{description}
\end{definition}

As observed by Baez and Melliès \cite{baez-2010-theories_with_dualities}
and Shulman \cite{shulman-2010-extraordinary_multicategories},
and explicated by Willerton \cite{willerton-2025-extranaturality},
the structure of an extranatural transformation has an equivalent formulation
as a natural transformation between profunctors
$
	\prohom{\product{\cat{A} , \op{\cat{A}} , \cat{B}}}{\product{\cat{B} , \op{\cat{C}} , \cat{C}}}
$;
which is to say, between functors
$
	\hom{\product{\op{\cat{A}} , \cat{A} , \op{\cat{B}} , \cat{B} , \op{\cat{C}} , \cat{C}}}{\Cat{Set}}
$,
with components
$$
	\centerdisplay{
	α \tuple*{A , A′ , B , B′ , C , C′} :
	\hom
		{(\product{\hom[\cat{A}]{A}{A′} , \hom[\cat{B}]{B}{B′} , \hom[\cat{C}]{C}{C′}})}
		{(\hom[\cat{D}]{F \tuple*{A , A′ , B}}{G \tuple*{B′ , C , C′}})}
	}
$$
This, in turn, can be interpreted as a $\Cat{Cat}$-square with the following boundary.
$$

$$
These observations inspire the concept of an extranatural square in a monoidal double category with adjoint folds.

\begin{definition}[extranatural square]
	Let $\cat{C}$ be a monoidal double category with \refer{adjoint folds}.
	For $\cat{C}$-arrows $f : \hom{\tensor{A , \dual{A} , B}}{D}$ and $g : \hom{\tensor{B , \dual{C} , C}}{D}$
	a $\cat{C}$-square $α$ is \define[extranatural square]{extranatural}
	if it is of the form
	$
		α : \doublehom{\tensor{ε A , ι B , η C}}{\procomp{} D}{f}{g}
	$,
	where $η$, $ε$, and $ι$ are the \refer{cap}, \refer{cup}, and \refer{strip}
	\refer[proarrow-dimension dinatural transformation]{dinatural transformations}.
\end{definition}

This allows us to characterize an \refer{extranatural transformation}
as an extranatural $\Cat{Cat}$-square.
We obtain its \refer[extranatural transformation object-component arrow]{object-component arrows} $α \tuple*{A , B , C}$
by precomposition with global elements,
$$

$$
where in the surface diagram we label a global element $\op{A} : \hom{\op{\cat{\product{}}}}{\op{\cat{A}}}$
by the corresponding object $A : \op{\cat{A}}$.
We obtain
domain extranaturality in $\cat{A}$ \eqref{extranatural transformation domain extranaturality},
ordinary naturality in $\cat{B}$ \eqref{extranatural transformation ordinary naturality}, and
codomain extranaturality in $\cat{C}$ \eqref{extranatural transformation codomain extranaturality},
$$
	%
$$
by
naturality for squares \eqref{proarrow-dimension dinatural transformation square naturality}
of the \refer[proarrow-dimension dinatural transformation]{dinatural transformations} $ε$ , $ι$, and $η$, respectively;
interpreting an arrow of a category $\cat{A}$
as a global element of the identity profunctor
$\procomp{} \, \cat{A} :  \prohom{\cat{A}}{\cat{A}} ≔ \hom[\cat{A}]{\arg}{\arg} : \hom{\product{\op{\cat{A}} , \cat{A}}}{\Cat{Set}}$.

Of course, $\Cat{Cat}$ has far more structure than a bare monoidal double category.
It is both a proarrow equipment \cite{grandis-2004-adjoint_for_double_categories, shulman-2008-framed_bicategories}
and cartesian monoidal \cite{aleiferi-2018-cartesian_double_categories}, hence symmetric monoidal.
Likewise, the functor $\op{} : \hom{\co{\Cat{Cat}}}{\Cat{Cat}}$ is a strict isomorphism,
which we can regard to be monoidal ($\op{(\product{\cat{A} , \cat{B}})} = \product{\op{\cat{A}} , \op{\cat{B}}}$),
or more conveniently,
reversing monoidal ($\op{(\product{\cat{A} , \cat{B}})} = \product{\op{\cat{B}} , \op{\cat{A}}}$),
so that the folds do not require braidings.
$$

$$
Such additional structure, which we make no attempt to account for here,
makes extranaturality in $\Cat{Cat}$ more complex.
For now we content ourselves to merely characterize extranatural transformations
using the machinery of dinaturality for double categories.
This may help facilitate the study of extranaturality in monoidal double categories other than $\Cat{Cat}$.

\paragraph{Related constructions}
The concept of extranaturality for $1$-categories was identified by Eilenberg and Kelly \cite{eilenberg-1966-functorial_calculus}.
Their original presentation was ``monolithic'', in terms of ``(generalized) natural'' families of maps between
functors $\hom{\product{\cat{A} , \cat{B} , \op{\cat{B}}}}{\cat{D}}$
and $\hom{\product{\cat{A} , \op{\cat{C}} , \cat{C}}}{\cat{D}}$, which are
``ordinary natural'' in $\cat{A}$ and ``extraordinary natural'' in $\cat{B}$ and $\cat{C}$.
They showed that extranatural transformations can be composed,
provided that the string diagram comprising the composite domain profunctor (a.k.a. \define{Eilenberg--Kelly graph}) is loop free.

Dubuc and Street decomposed Eilenberg and Kelly's generalized naturality into three constituent parts:
ordinary naturality, codomain extranaturality, and domain extranaturality;
and showed that transformations of each sort are instances of $1$-categorical dinatural transformations \cite{street-1970-dinatural_transformations}.
Street generalized the concept of dinaturality to the setting of arbitrary monoidal bicategories with duals,
where he observed that Eilenberg--Kelly graphs arise as subdiagrams of string diagram boundaries
of surface diagrams representing extranatural $2$-cells \cite{street-2003-functorial_calculus}.

Willerton studied extranaturality,
first in the context of a monoidal bicategorory with duals \cite{willerton-2008-diagrammatic_hopf_monads} (where he called it ``dinaturality''),
and then in a monoidal double category with duals \cite{willerton-2025-extranaturality}.
He extended Street's surface diagram representation of extranatural $2$-cells into a full surface diagram calculus,
including the diagrammatic presentation of the extranaturality laws shown above.

The double categorical treatment of mates was introduced by Kelly and Street \cite{kelly-1974-2_categories}.
It was elaborated by Cheng, Gurski, and Riehl to explicitly account for the action of the opposite category duality
in the double category of categories, adjunctions, functors, and natural transformations \cite{riehl-2014-multivariable_adjunctions}.
The main difference in our presentation above
is that we don't assume the functor $\dual{} : \hom{\co{\cat{C}}}{\cat{C}}$
(specializing to $\op{} : \hom{\co{\Cat{Cat}}}{\Cat{Cat}}$) to be an isomorphism.
Instead, we assume the structure of adjoint folds,
from which we derive some necessary properties of this functor.

In unpublished notes, Baez and Melliès describe characterizing $1$-categorical extranaturality
in the double category of small categories \cite{baez-2010-theories_with_dualities}.
Shulman has proposed compact (closed) double categories as a venue to study extranaturality more generally
\cite{shulman-2010-extraordinary_2multicategories, shulman-2010-extraordinary_multicategories, nlab-compact_double_category}.

%% file: content/parts/conclusion.tex

In this paper we extended the concept of dinaturality to the setting of double categories
and observed that the classical results about dinaturality in $1$-categories generalize nicely there.
We showed how the $2$-dimensional algebra of dinatural transformations and dimodifications of double categories
can be understood as the projection of a simpler $3$-dimensional algebra,
extending the case for ordinary natural transformations and modifications of double categories.
Finally, we showed how extranatural transformations in the cartesian monoidal double category of small categories
can be generalized to extranatural squares in monoidal double categories with adjoint folds.
Throughout, we used the graphical language of surface diagrams
to give topologically intuitive representations
of terms that can be quite unwieldy when written out in a $1$-dimensional notation.
Identifying the boundaries of such terms from their diagrams is trivial,
while calculating them from their linear representations can be tedious and error-prone.

There are some obvious tangents that we could have pursued, but didn't.
For example, we assumed our double categories to be strictly monoidal.
This was mostly for the sake of expediency, because in order to capture our motivating example double category $\Cat{Cat}$,
we already needed to account for the coherences of weak proarrow-dimension composition structure,
and the thought of involving monoidal coherences as well did not bring joy.
Of course, one option would be to generalize to weakly monoidal double categories,
where the product is associative and unital only up to coherent isomorphism;
but perhaps a more interesting generalization would be to Gray monoidal double categories,
where the monoidal structure is dimension-summing rather than dimension-sup'ing
\cite{bohm-2019-monoidal_product_of_double_categories, morehouse-2023-cartesian_gray_monoidal_double_categories}.
In this setting, the equalities we have justified by ``monoidal naturality'' would instead be mediated by Gray interchangers.

Similarly, we chose to consider \emph{strict} functors of weak double categories,
whereas \emph{pseudo} functors are arguably a more natural choice.
These preserve the weak composition structure in the proarrow dimension only up to coherent comparitors.
Our choice here was again made to avoid the proliferation of coherence cells that are not germane to our topic of study, dinaturality.

We only scratched the surface of a dinaturality-based exploration of extranaturality.
Extranaturality in $\Cat{Cat}$ is complicated by the fact that $\Cat{Cat}$ has so much additional structure
when regarded as a monoidal double category with a functor from its proarrow-dimension reflection.
It would be worthwhile to try to better understand how its various features
interact with extranatural transformations
and perhaps generalize some of these to abstract definitions
(e.g., extranatural square in a symmetric monoidal double category with adjoint folds,
or extranatural square in a monoidal double category with a strict duality).

Having described the basic concepts of dinaturality for double categories,
it would be nice to identify further examples of structures that can be characterized using them.

%% file: content/parts/proofs.tex

For the sake of diagram legibility,
in the following proofs we employ the \refer{interchanger} notation described in section \ref{section: diagrammatics}.
To determine the identity of an unlabeled $2$-cell, simply consult the $1$-cells on its boundary.


\begin{proof}[proof of lemma \ref{theorem: natural transformation product yang baxter}] \hypertarget{proof: natural transformation product yang baxter}{}
	By the introduction and cancellation of consecutive inverse component disks and
	natural transformation naturality for squares \eqref{proarrow-dimension natural transformation square naturality}
	for the identity disk $\comp{} \, \tuple*{p , m}$
	we have:
	$$
		\phantom{\; \overset{\eqref{proarrow-dimension natural transformation square naturality}}{≅} \;}

	$$
	$$
		\; \overset{\eqref{proarrow-dimension natural transformation square naturality}}{=} \;
		%
	$$
\end{proof}

\begin{proof}[proof of lemma \ref{theorem: dummy natural transformations}] \hypertarget{proof: dummy natural transformations}{}
	\eqref{proarrow-dimension natural transformation arrow composition preservation},
	\eqref{proarrow-dimension natural transformation proarrow composition compatibility}, and
	\eqref{proarrow-dimension natural transformation square naturality} for $\dummy{α}$
	follow from the analogous property for $α$.
	
	For example,
	in the case of binary \eqref{proarrow-dimension natural transformation proarrow composition compatibility} we have:
	$$
\begin{array}{cl}
			&
			\dummy{α} (\procomp{\tuple*{p , m} , \tuple*{q , n}})
				\; = \;
			\dummy{α} \tuple*{(\procomp{p , q}) , (\procomp{m , n})}
				\; ≔ \;
			α (\procomp{m , n})
			\\
				\overset{\eqref{proarrow-dimension natural transformation proarrow composition compatibility}}{≅}
				&
			\comp{(\procomp{F m , α n}) , ({\procomp{α m , G n}})}
				\; ≕ \;
			\comp{(\procomp{\dummy{F} \tuple*{p , m} , \dummy{α} \tuple*{q , n}}) , (\procomp{\dummy{α} \tuple*{p , m} , \dummy{G} \tuple*{q , n}})}
			\\
		\end{array}
	$$
	and in the case of \eqref{proarrow-dimension natural transformation square naturality} we have:
	$$
		\centerdisplay{
		%
		\, \overset{\eqref{proarrow-dimension natural transformation square naturality}}{=} \,
		%
		}
	$$
	Likewise, \eqref{modification arrow naturality} for $\dummy{μ}$ follows from the analogous property for $μ$:
	$$
		\centerdisplay{
		%
		\, \overset{\eqref{modification arrow naturality}}{≅} \,
		%
		}
	$$
	and similarly for \eqref{modification proarrow naturality}.
\end{proof}


\begin{proof}[proof of proposition \ref{theorem: dummy difunctors are functors}] \hypertarget{proof: dummy difunctors are functors}{}
	The \refer{dummy functor} construction,
	sending a functor $F : \hom{\cat{C}}{\cat{D}}$
	to the dummy functor $\dummy{F} ≔ \comp{π_1 , F} : \hom{\product{\cat{B} , \cat{C}}}{\cat{D}}$,
	is invertible unless $\cat{B} = \cat{\coproduct{}}$ and $\cat{C} ≠ \cat{\coproduct{}}$.
	Otherwise, we recover $F$ as $\dummy{F} \tuple*{B , \arg}$ for any $B : \cat{B}$.
	Of course, $\co{\cat{C}} = \cat{\coproduct{}}$ iff $\cat{C} = \cat{\coproduct{}}$.
\end{proof}

\begin{proof}[proof of proposition \ref{theorem: transformation diagonalization}] \hypertarget{proof: transformation diagonalization}{}
	\eqref{proarrow-dimension dinatural transformation arrow composition preservation},
	\eqref{proarrow-dimension dinatural transformation proarrow composition compatibility}, and
	\eqref{proarrow-dimension dinatural transformation square naturality} for $\di{α}$
	follow from their analogues
	\eqref{proarrow-dimension natural transformation arrow composition preservation},
	\eqref{proarrow-dimension natural transformation proarrow composition compatibility}, and
	\eqref{proarrow-dimension natural transformation square naturality} for $α$,
	together with \eqref{natural transformation product yang baxter}.
	The most interesting cases are
	compatibility with binary proarrow composition
	\eqref{proarrow-dimension dinatural transformation proarrow composition compatibility}
	and naturality for squares
	\eqref{proarrow-dimension dinatural transformation square naturality}.
	
	For binary \eqref{proarrow-dimension dinatural transformation proarrow composition compatibility},
	for consecutive proarrows $m : \prohom{A}{B}$ and $n :  \prohom{B}{C}$
	we have:
	$$
		\centerdisplay{
		\di{α} \tuple*{\procomp{m , n}}
			\, :≅ \,

		\, \overset{\eqref{proarrow-dimension natural transformation proarrow composition compatibility}}{≅} \,
		%
		}
	$$
	$$
			\, \overset{\eqref{natural transformation product yang baxter}}{≅} \,
		%
	$$
	
	For \eqref{proarrow-dimension dinatural transformation square naturality},
	for square $φ : \doublehom[(\doublehom{\prohom{A}{C}}{\prohom{B}{D}}{\hom{A}{B}}{\hom{C}{D}})]{m}{n}{f}{g}$
	we have:
	$$
		%
	$$
	$$
		\, \overset{\eqref{proarrow-dimension natural transformation square naturality}}{=} \,
		%
	$$
	$$
		\, \overset{\eqref{proarrow-dimension natural transformation square naturality}}{=} \,
		%
	$$
	where in the unannotated step, both central triangles are equal to $G \tuple*{φ , φ}$.
\end{proof}

\begin{proof}[proof of corollary \ref{theorem: dummy dinatural transformations}] \hypertarget{proof: dummy dinatural transformations}{}
	Simply \refer[transformation diagonalization]{diagonalize} the \refer{dummy natural transformation}
	$\dummy{α} : \prohom[\hom[]{\product{\co{\cat{C}} , \cat{C}}}{\cat{D}}]{\dummy{F}}{\dummy{G}}$.
\end{proof}

\begin{proof}[proof of proposition \ref{theorem: dummy dinatural transformations are natural transformations}] \hypertarget{proof: dummy dinatural transformations are natural transformations}{}
	Given a dinatural $α : \prohom<dinatural>{\dummy{F}}{\dummy{G}}$
	we define a purportedly natural $α′ : \prohom{F}{G}$
	with components $α′ A ≔ α A$, $α′ f ≔ α f$, and $α′ m :≅ α m$.
	That $′$ is inverse to $\di{\dummy{}}$ is immediate, and
	\eqref{proarrow-dimension natural transformation arrow composition preservation},
	\eqref{proarrow-dimension natural transformation proarrow composition compatibility}, and
	\eqref{proarrow-dimension natural transformation square naturality} for $α′$
	follow from
	\eqref{proarrow-dimension dinatural transformation arrow composition preservation},
	\eqref{proarrow-dimension dinatural transformation proarrow composition compatibility}, and
	\eqref{proarrow-dimension dinatural transformation square naturality} for $α$.
	The most interesting cases are again
	compatibility with binary proarrow composition
	\eqref{proarrow-dimension natural transformation proarrow composition compatibility}
	and naturality for squares
	\eqref{proarrow-dimension natural transformation square naturality}.

	For binary \eqref{proarrow-dimension natural transformation proarrow composition compatibility},
	for consecutive proarrows $m : \prohom{A}{B}$ and $n :  \prohom{B}{C}$
	we have:
	$$

	$$
	
	$$
		\, \overset{\eqref{proarrow-dimension dinatural transformation proarrow composition compatibility}}{≅} \,
		%
	$$
	
	For \eqref{proarrow-dimension natural transformation square naturality},
	for square $φ : \doublehom[(\doublehom{\prohom{A}{C}}{\prohom{B}{D}}{\hom{A}{B}}{\hom{C}{D}})]{m}{n}{f}{g}$
	we have:
	$$
			\phantom{≅} \quad
		%
	$$
	$$
		\overset{\eqref{proarrow-dimension dinatural transformation square naturality}}{=} \quad
		%
	$$
\end{proof}

\begin{proof}[proof of proposition \ref{theorem: groupoidal dinatural composition}] \hypertarget{proof: groupoidal dinatural composition}{}
	The proarrow-component disks are probably easier to understand in the surface diagram notation of section \ref{section: diagrammatics}.
	$$
		\centerdisplay{
		%
		}
	$$
	\eqref{proarrow-dimension dinatural transformation arrow composition preservation} and
	nullary \eqref{proarrow-dimension dinatural transformation proarrow composition compatibility}
	for $\procomp{α , β}$ are straightforward.
	Binary \eqref{proarrow-dimension dinatural transformation proarrow composition compatibility}
	amounts to the following fact.
	$$
		%
	$$
	\eqref{proarrow-dimension dinatural transformation square naturality}
	uses that same property for $α$ and for $β$,
	together with the introduction and cancellation of consecutive proarrow-dimension inverse $\cat{C}$-squares
	$φ : \doublehom{m}{n}{f}{g}$ and $\inv{φ} : \doublehom{\inv{m}}{\inv{n}}{g}{f}$.
	$$
		\centerdisplay{
		\phantom{\; \overset{\eqref{proarrow-dimension dinatural transformation square naturality}}{=} \;}
		%
		}
	$$
	$$
		\centerdisplay{
		\; \overset{\eqref{proarrow-dimension dinatural transformation square naturality}}{=} \;
		%
		}
	$$
\end{proof}

\begin{proof}[proof of proposition \ref{theorem: dinatural-natural composition}] \hypertarget{proof: dinatural-natural composition}{}
	\eqref{proarrow-dimension dinatural transformation arrow composition preservation},
	\eqref{proarrow-dimension dinatural transformation proarrow composition compatibility}, and
	\eqref{proarrow-dimension dinatural transformation square naturality}
	for $\procomp{α , β}$
	follow from
	\eqref{proarrow-dimension dinatural transformation arrow composition preservation},
	\eqref{proarrow-dimension dinatural transformation proarrow composition compatibility}, and
	\eqref{proarrow-dimension dinatural transformation square naturality}
	for $α$
	and
	\eqref{proarrow-dimension natural transformation arrow composition preservation},
	\eqref{proarrow-dimension natural transformation proarrow composition compatibility}, and
	\eqref{proarrow-dimension natural transformation square naturality}
	for $β$
	together with \eqref{natural transformation product yang baxter}.
	The most interesting cases are again
	compatibility with binary proarrow composition
	\eqref{proarrow-dimension dinatural transformation proarrow composition compatibility}
	and naturality for squares
	\eqref{proarrow-dimension dinatural transformation square naturality}.
	
	For binary \eqref{proarrow-dimension dinatural transformation proarrow composition compatibility},
	for consecutive proarrows $m : \prohom{A}{B}$ and $n :  \prohom{B}{C}$
	we have:
	$$
		\centerdisplay{
		(\procomp{α , β}) \tuple*{\procomp{m , n}}
		\; :≅ \hspace{-1em}

			\hspace{-1em} \overset
			{
				\eqref{proarrow-dimension dinatural transformation proarrow composition compatibility} ,
				\eqref{proarrow-dimension natural transformation proarrow composition compatibility}
			}{≅} \,
		%
		}
	$$
	$$
			\, \overset{\eqref{natural transformation product yang baxter}}{≅} \,
		%
	$$
	
	For \eqref{proarrow-dimension dinatural transformation square naturality},
	for square $φ : \doublehom[(\doublehom{\prohom{A}{C}}{\prohom{B}{D}}{\hom{A}{B}}{\hom{C}{D}})]{m}{n}{f}{g}$
	we have:
	$$
		%
	$$
	$$
		\, \overset{\eqref{proarrow-dimension natural transformation square naturality}}{=} \,
		%
		\, \overset{\eqref{proarrow-dimension dinatural transformation square naturality}}{=} \,
		%
	$$
	$$
		\, \overset{\eqref{proarrow-dimension natural transformation square naturality}}{=} \,
		%
	$$
\end{proof}

\begin{proof}[proof of proposition \ref{theorem: dinatural-natural composition laws}] \hypertarget{proof: dinatural-natural composition laws}{}
	By direct computation.
\end{proof}

\begin{proof}[proof of proposition \ref{theorem: modification diagonalization}] \hypertarget{proof: modification diagonalization}{}
	Naturality for arrows \eqref{dimodification arrow naturality}
	and for proarrow \eqref{dimodification proarrow naturality}
	of $\di{μ}$
	follow from the corresponding property \eqref{modification arrow naturality}
	or \eqref{modification proarrow naturality}
	of $μ$.
	In the more interesting case, \eqref{dimodification proarrow naturality}, for a proarrow $m : \prohom{A}{B}$ we have:
	$$

	$$
	$$
		\, \overset{\eqref{modification proarrow naturality}}{=} \,
		%
	$$
	$$
		\, \overset{\eqref{modification proarrow naturality}}{=} \,
		%
	$$
	where in the unannotated step, both central triangles are equal to $α \tuple*{m , m}$.
\end{proof}

\begin{proof}[proof of corollary \ref{theorem: dummy dimodifications}] \hypertarget{proof: dummy dimodifications}{}
	Simply \refer[modification diagonalization]{diagonalize} the \refer{dummy modification}
	$
		\dummy{μ} : \doublehom[\doublehom[\hom[]{\product{\co{\cat{C}} , \cat{C}}}{\cat{D}}]
		{\prohom{\dummy{F}}{\dummy{I}}}{\prohom{\dummy{G}}{\dummy{J}}}{\hom{\dummy{F}}{\dummy{G}}}{\hom{\dummy{I}}{\dummy{J}}}]
		{\dummy{γ}}{\dummy{δ}}{\dummy{α}}{\dummy{β}}
	$.
\end{proof}

\begin{proof}[proof of proposition \ref{theorem: dummy dimodifications are modifications}] \hypertarget{proof: dummy dimodifications are modifications}{}
	Given a dimodification $μ : \doublehom{\di{\dummy{γ}}}{\di{\dummy{δ}}}{\dummy{α}}{\dummy{β}}$
	we define a purported modification $μ′ : \doublehom{γ}{δ}{α}{β}$
	with components $μ′ A ≔ μ A$.
	That $′$ is inverse to $\di{\dummy{}}$ is immediate, and
	\eqref{modification arrow naturality} and
	\eqref{modification proarrow naturality} for $μ′$
	follow from
	\eqref{dimodification arrow naturality} and
	\eqref{dimodification proarrow naturality} for $μ$.
	For example, we obtain \eqref{modification arrow naturality} for $μ′$ as
	$$

		\, \overset{\eqref{dimodification arrow naturality}}{≅} \,
			%
	$$
\end{proof}


\begin{proof}[proof of proposition \ref{theorem: zigzag natural transformations}] \hypertarget{proof: zigzag natural transformations}{}
	\eqref{proarrow-dimension natural transformation arrow composition preservation},
	\eqref{proarrow-dimension natural transformation proarrow composition compatibility}, and
	\eqref{proarrow-dimension natural transformation square naturality}
	 for $S$ and $Z$
	follow from their analogues
	\eqref{proarrow-dimension dinatural transformation arrow composition preservation},
	\eqref{proarrow-dimension dinatural transformation proarrow composition compatibility}, and
	\eqref{proarrow-dimension dinatural transformation square naturality}
	for $η$, $ε$, $ι$, and $\dual{ι}$.
	The most interesting cases are again
	compatibility with binary proarrow composition
	\eqref{proarrow-dimension natural transformation proarrow composition compatibility}
	and naturality for squares
	\eqref{proarrow-dimension natural transformation square naturality}.
	
	Binary \eqref{proarrow-dimension natural transformation proarrow composition compatibility} for $S$
	follows from binary \eqref{proarrow-dimension dinatural transformation proarrow composition compatibility}
	for $η$ and $ε$,
	$$
		\centerdisplay{
		S (\procomp{m , n})
		\; ≔ \;

		\; \overset{\eqref{proarrow-dimension dinatural transformation proarrow composition compatibility}}{≅} \;
		%
		\; ≕ \;
		\comp{(\procomp{m , S n}) , (\procomp{S m , n})}
	$$
	where in the unannotated step the central portions of the diagrams are equal by monoidal naturality.
	In particular, if $m : \prohom{A}{B}$ and $n : \prohom{B}{C}$ then they are both equal to
	$\comp{(\procomp{(\tensor{m , η n}) , (\tensor{ε B , C})}) , (\procomp{({\tensor{A , η B}}) , (\tensor{ε m , n})})}$.
	
	\eqref{proarrow-dimension natural transformation square naturality} for $S$
	follows from \eqref{proarrow-dimension dinatural transformation square naturality}
	for the sequence of dinatural transformations ${ι , ε , η , ι}$.
	These let us ``slide'' a square $φ : \doublehom{m}{n}{f}{g}$
	across a zigzag as follows.
	$$
		\centerdisplay{

		\hspace{-2mm} \overset{\eqref{proarrow-dimension dinatural transformation square naturality} ι , ε}{=} \hspace{-2mm}
		%
		\hspace{-2mm} \overset{\eqref{proarrow-dimension dinatural transformation square naturality} η , ι}{=} \hspace{-2mm}
		%
		}
	$$
\end{proof}

\begin{proof}[proof of proposition \ref{theorem: sesqui-zigzag dinatural transformations}] \hypertarget{proof: sesqui-zigzag dinatural transformations}{}
	\eqref{proarrow-dimension dinatural transformation arrow composition preservation},
	\eqref{proarrow-dimension dinatural transformation proarrow composition compatibility}, and
	 \eqref{proarrow-dimension dinatural transformation square naturality}
	 for $M$ and $W$ follow from the same properties of $η$ and $ε$.
\end{proof}

\begin{proof}[proof of lemma \ref{theorem: mate bijection}] \hypertarget{proof: mate bijection}{}
	We show
	$
		\comp{\leftmate{\arg} , \rightmate{\arg}} =
		\comp{} (\doublehom[\doublehom[]{\prohom{A}{C}}{\prohom{B}{D}}{\hom{A}{B}}{\hom{C}{D}}]{m}{n}{f}{g})
	$.
	The other case is dual.
	The ragged surface diagram for $\rightmate{(\leftmate{φ})}$ is shown.
	$$

	$$
	Using \refer{monoidal naturality} we reorder the folds so that the front cap is below the back cap and the back cup above the front cup,
	creating two $S$-zigzags that overlap in the proarrow dimension only in a band about $φ$.
	$$
		%
	$$
	Moreover, the squares $S f$, $S g$, $\comp{} \, S B$ and $\comp{} \, S C$
	need not  overlap the other $S$-zigzag in the proarrow dimension at all.
	This gives the following block decomposition of the projection string diagram,
	where the folds have been elided for clarity and the cusps have been made explicit.
	$$
		%
	$$
	Using modification naturality for proarrows \eqref{modification proarrow naturality}
	for the cusps we slip $m$ and $n$ off the $S$-zigzags (and absorb the units).
	$$
		%
	$$
	Using modification naturality for arrows \eqref{modification arrow naturality}
	we slip $f$ and $g$ off the $S$-zigzags as well.
	Finally, we cancel the consecutive inverse cusps, leaving $φ$.
\end{proof}

\begin{proof}[proof of lemma \ref{theorem: mates and *}] \hypertarget{proof: mates and *}{}
	For $\dual{φ} = \leftmate{φ}$,
	the ragged surface diagram for $\leftmate{φ}$ is shown on the left.
	Using dinatural transformation naturality for squares \eqref{proarrow-dimension dinatural transformation square naturality}
	for $ε$ we flip $φ$ to the front surface as $\dual{φ}$.
	This results in proarrow and arrow $Z$-zigzags, as shown on the right.
	Using modification naturality for arrows \eqref{modification arrow naturality} and for proarrows \eqref{modification proarrow naturality}
	we slip $Z f$ and $Z m$ past the (elided) $ξ$-cusp as $\procomp{} \, \dual{f}$ and $\comp{} \, \dual{m}$.
	Finally, we cancel the consecutive inverse cusps, leaving $\dual{φ}$.
	The case of $φ = \rightmate{(\dual{φ})}$ is dual.
	$$

	$$
	For $ψ = \dual{(\rightmate{ψ})}$,
	we know that $ψ = \leftmate{(\rightmate{ψ})} : \doublehom{\dual{m}}{\dual{n}}{\dual{g}}{\dual{f}}$ by lemma \ref{theorem: mate bijection},
	and that $\dual{(\rightmate{ψ})} = \leftmate{(\rightmate{ψ})} : \doublehom{\dual{m}}{\dual{n}}{\dual{g}}{\dual{f}}$ by the above.
	The result follows by symmetry and transitivity.
\end{proof}

\begin{proof}[proof of lemma \ref{theorem: mates of identity disks}] \hypertarget{proof: mates of identity disks}{}
	These follow from modification naturality for  arrows \eqref{modification arrow naturality}
	and for proarrows \eqref{modification proarrow naturality} of the \refer{cusp} modifications.
\end{proof}

\begin{proof}[proof of lemma \ref{theorem: mates of proarrow-dimension composite squares}] \hypertarget{proof: mates of proarrow-dimension composite squares}{}
	The ragged surface diagram for $\procomp{\leftmate{ψ} , \leftmate{φ}}$ is shown.
	$$

	$$
	Using modification naturality for proarrows \eqref{modification proarrow naturality}
	we slip the proarrows $\dual{p}$ and $\dual{n}$ past their respective (elided) $ξ$-cusps onto the $Z$-zigzags.
	Then using \refer{monoidal naturality} we reorder the folds so that the front cap is below $φ$ and the back cup is above $ψ$,
	creating an $S$-zigzag in the center of the diagram.
	$$
		%
	$$
	Using modification naturality for arrows \eqref{modification arrow naturality} and for proarrows \eqref{modification proarrow naturality}
	we slip the proarrows and the arrow off the central $S$-zigzag past the (elided) $ζ$-cusp.
	Finally, we cancel the consecutive inverse $ζ$-cusps, leaving $\leftmate{(\procomp{φ , ψ})}$.
	$$
		%
	$$
	The case of right mates is dual.
\end{proof}

\begin{proof}[proof of lemma \ref{theorem: mates of arrow-dimension composite squares}] \hypertarget{proof: mates of arrow-dimension composite squares}{}
	The ragged surface diagram for $\comp{\leftmate{φ} , \leftmate{ψ}}$ is shown.
	$$
		%
	$$
	Using dinatural transformation naturality for squares \eqref{proarrow-dimension dinatural transformation square naturality}
	we flip $φ$ to the back surface as $\dual{φ}$.
	This results in a $Z$-zigzag for the proarrow $n$ on the left side of the diagram.
	$$
		%
	$$
	Using modification naturality for proarrows \eqref{modification proarrow naturality}
	we slip this zigzag past the $ξ$-cusp.
	$$
		%
	$$
	Then we cancel the consecutive inverse cusps and flip $\dual{φ}$ back to the middle surface as $φ$,
	leaving $\leftmate{(\comp{φ , ψ})}$.
	$$
		%
	$$
	
	The case of right mates is almost dual,
	except that in the fist step, when we use dinatural transformation naturality for squares \eqref{proarrow-dimension dinatural transformation square naturality}
	to flip $φ : \doublehom{\dual{m}}{\dual{n}}{\dual{g}}{\dual{f}}$ to the back surface,
	we need to know that $φ = \dual{θ}$ for some $θ : \doublehom{m}{n}{f}{g}$.
	But by lemma \ref{theorem: mates and *} we know $φ = \dual{(\rightmate{φ})}$.
\end{proof}

\begin{proof}[proof of proposition \ref{theorem: * functor fully faithful on squares}] \hypertarget{proof: * functor fully faithful on squares}{}
	Given $ψ : \doublehom{\dual{m}}{\dual{n}}{\dual{g}}{\dual{f}}$,
	by lemma \ref{theorem: mate bijection}, $\rightmate{ψ} : \doublehom{m}{n}{f}{g}$
	is the only $φ :  \doublehom{m}{n}{f}{g}$ such that $\leftmate{φ} = ψ$.
	So by lemma \ref{theorem: mates and *}, $\rightmate{ψ} : \doublehom{m}{n}{f}{g}$
	is the only $φ :  \doublehom{m}{n}{f}{g}$ such that $\dual{φ} = ψ$.
\end{proof}

\begin{proof}[proof of lemma \ref{theorem: same * 1}] \hypertarget{proof: same * 1}{}
	For objects,
	the composite
	$\procomp{S \tuple*{A_1 , A_0} , S \tuple*{A_0 , A_1}}$
	is shown on the left.
	Using \refer{monoidal naturality} we reorder the folds to create a zigzag component $Z A_0$ in the center of the diagram.
	We then use the cusp $ξ A_0$ to straighten this zigzag
	and the cusp $ζ A_1$ to straighten the remaining zigzag $S A_1$.
	$$

	$$
	For arrows,
	the composite
	$\procomp{S \tuple*{f_1 , f_0} , S \tuple*{f_0 , f_1}}$
	is shown on the left.
	Using monoidal naturality we reorder the folds to create a zigzag component $Z f_0$ in the center of the diagram
	as shown on the right.
	Conjugating this square by the nested folding coherators described above,
	we slip the arrow past successive cusps and cancel consecutive inverses
	leaving $\procomp{} \, f_1$.
	$$
		%
	$$
	For proarrows,
	the composite
	$
		\comp
		{
			(\procomp{S \tuple*{m_1 , m_0} , \comp{} \, S \tuple*{B_0 , B_1}})
			,
			(\procomp{\comp{} \, S \tuple*{A_1 , A_0} , S \tuple*{m_0 , m_1}})
		}
	$
	is shown on the left.
	Using monoidal naturality we reorder the folds to create a zigzag component $Z m_0$ in the center of the diagram
	as shown on the right.
	Conjugating this square by the nested folding coherators described above,
	we slip the proarrow past successive cusps and cancel consecutive inverses
	leaving $\comp{} \, m_1$.
	$$
		%
	$$
\end{proof}

\begin{proof}[proof of proposition \ref{theorem: same * 2}] \hypertarget{proof: same * 2}{}
	By naturality for squares \eqref{proarrow-dimension dinatural transformation square naturality}
	of the dinatural transformations $ε$ and $η$
	we have
	$$
		%
	$$
	Making explicit the folding coherators along the diagram edges,
	by lemma \ref{theorem: same * 1} we have
	$$
		%
	$$
\end{proof}

\begin{proof}[proof of proposition \ref{theorem: adjoint folds in Cat}] \hypertarget{proof: adjoint folds in Cat}{}
	For arrow composition preservation \eqref{proarrow-dimension dinatural transformation arrow composition preservation} we have
	for category $\cat{A}$
	$$
		η (\comp{} \, \cat{A})
		:≅ \procomp{} \, (\comp{} \, \cat{A})
		= \comp{} \, (\procomp{} \, \cat{A})
		≅: \comp{} \, η \cat{A}
	$$
	and for consecutive functors $F : \hom{\cat{A}}{\cat{B}}$ and $G : \hom{\cat{B}}{\cat{C}}$
	$$
		η (\comp{F , G})
		:≅ \procomp{} \, (\comp{F , G})
		= \comp{\procomp{} \, F , \procomp{} \, G}
		≅: \comp{η F , η G}
	$$
	For proarrow composition compatibility \eqref{proarrow-dimension dinatural transformation proarrow composition compatibility} we have
	for category $\cat{A}$
	$$
		η (\procomp{} \, \cat{A})
		:≅ \comp{} \, (\procomp{} \, \cat{A})
		≅: \comp{} \, η \cat{A}
	$$
		and for consecutive profunctors $M : \prohom{\cat{A}}{\cat{B}}$ and $N : \prohom{\cat{B}}{\cat{C}}$
	$$
		η (\procomp{M , N})
		:≅ \comp{} \, (\procomp{M , N})
		≅ \comp{\comp{} \, (\procomp{M , N , \procomp{} \, \cat{C}}) , \comp{} \, (\procomp{\procomp{} \, \cat{A} , M , N})}
	$$
	$$
		≅: \comp{(\procomp{η N , \comp{} \, (\product{\op{M} , \procomp{} \, \cat{C}})}) , (\procomp{η M , \comp{} \, (\product{\procomp{} \, \op{\cat{A}} , N})})}
	$$
	For square naturality \eqref{proarrow-dimension dinatural transformation square naturality} we have
	for natural transformation $Φ : \doublehom{M}{N}{F}{G}$
	$$
		\comp{(\procomp{η G , (\product{\op{Φ} , \procomp{} \, G})}) , η N}
		:≅ \comp{(\procomp{Φ , \procomp{} \, G , \procomp{} \, G}) , \comp{} \, N}
		≅ Φ
	$$
	$$
		≅ \comp{\comp{} \, M , (\procomp{\procomp{} \, F , \procomp{} \, F , Φ})}
		≅∶ \comp{η M , (\procomp{η F , (\product{\procomp{} \, \op{F} , Φ})})}
	$$
	And likewise for $ε$.
	
	The object components of the \refer{zigzag} natural transformations
	\[
		S : \prohom[\hom[]{\Cat{Cat}}{\Cat{Cat}}]{\comp{}}{\comp{}}
		\quad \text{and} \quad
		Z : \prohom[\hom[]{\co{\Cat{Cat}}}{\Cat{Cat}}]{\op{}}{\op{}}
	\]
	are isomorphic to identities with the \refer[cusp]{cusps} $ζ$ and $ξ$ witnessing these isomorphisms.
\end{proof}